\documentclass[11pt]{amsart}

\usepackage{accents}
\usepackage{appendix}
\usepackage{amsfonts}
\usepackage{amsmath}
\usepackage{amssymb}	
\usepackage{amsthm}
\usepackage{array,booktabs,multirow}
\usepackage{braket}
\usepackage{cite}
\usepackage{dsfont}
\usepackage[shortlabels]{enumitem}
\usepackage{etoolbox}
\usepackage{float}
\usepackage[hang, flushmargin]{footmisc}
\usepackage{latexsym}
\usepackage{lipsum}
\usepackage{needspace}
\usepackage{tikz}
\usepackage{hyperref}
\usepackage{url}
\usetikzlibrary{matrix,arrows}

\theoremstyle{plain}
\newtheorem{thm}{Theorem}[section]
\newtheorem{cor}[thm]{Corollary}
\newtheorem{lem}[thm]{Lemma}
\newtheorem{prop}[thm]{Proposition}

\theoremstyle{definition}

\theoremstyle{remark}

\AtBeginEnvironment{lem}{\Needspace{3\baselineskip}}%
\AtBeginEnvironment{thm}{\Needspace{3\baselineskip}}
\AtBeginEnvironment{prop}{\Needspace{3\baselineskip}}
\AtBeginEnvironment{cor}{\Needspace{3\baselineskip}}

\setlist[enumerate,1]{leftmargin=2.2em}

\def\A{\mathcal A}

\def\C{\mathbb C}
\def\R{\mathbb R}

\def\N{\mathbb N}
\def\F{\mathbb F}

\def\Z{\mathbb Z}
\def\sl_2{\mathfrak{sl}_2}
\def\U{U_q(\mathfrak{sl}_2)}

\def\e{\varepsilon}

\title[An embedding of the universal Askey-Wilson algebra into $\U\otimes \U\otimes \U$]
{An embedding of the universal Askey-Wilson algebra\\ into $\U\otimes \U\otimes \U$}

\author{Hau-Wen Huang}
\address{
Department of Mathematics\\
National Central University\\
Chung-Li 32001 Taiwan
}
\email{hauwenh@math.ncu.edu.tw}

\thanks{The research is supported by the Ministry of Science and Technology of Taiwan under the project MOST 105-2115-M-008-013.
}

\begin{document}
\begin{abstract}
The Askey--Wilson algebras were used to interpret the algebraic structure hidden in the Racah--Wigner coefficients of the quantum algebra $\U$. In this paper, we display an injection of a universal analog $\triangle_q$ of Askey--Wilson algebras into $\U\otimes \U\otimes \U$ behind the application.
Moreover we formulate the decomposition rules for $3$-fold tensor products of irreducible Verma $\U$-modules and of finite-dimensional irreducible $\U$-modules into the direct sums of finite-dimensional irreducible $\triangle_q$-modules. 
\end{abstract}

\maketitle


\section{Introduction}\label{s:intro}

Fix a nonzero complex scalar $q$ with $q^2\not=1$.
The quantum algebra $\U$ is a unital associative algebra over the complex number field $\C$ generated by $E$, $K^{\pm 1}$, $F$ subject to
\begin{gather*}
K K^{-1}=K^{-1}K=1,\\
KE=q^2EK,
\qquad
KF=q^{-2}FK, \\
EF-FE=\frac{K-K^{-1}}{q-q^{-1}}.
\end{gather*}
The Casimir element of $\U$ has the expression
\begin{gather*}
\Lambda=(q-q^{-1})^2 FE+qK+q^{-1} K^{-1}
\end{gather*}
up to scalar multiplication. Let $\Delta:\U\to \U\otimes \U$ denote the comultiplication of $\U$. The $n$-fold comultiplication $\Delta_n:\U\to \U^{\otimes (n+1)}$ is recurrently defined by $\Delta_0=1$ and 
\begin{gather*}
\Delta_n=(1^{\otimes (n-1)}\otimes \Delta)\circ \Delta_{n-1}
\qquad 
\hbox{for all integers $n\geq 1$}.
\end{gather*}

Consider a Hopf $*$-algebra of $\U$, for instance $U_q(\mathfrak{su}_2)$ with $q$ a real number. Let $V$ denote a $3$-fold tensor product of irreducible unitary $U_q(\mathfrak{su}_2)$-modules. The inner products between two coupled bases of $V$ are called the {\it Racah--Wigner coefficients} of $\U$.  
In fact, the two coupled bases of $V$ are the orthonormal eigenbases of 
\begin{gather*}
K_0=\Delta(\Lambda)\otimes 1,
\qquad
K_1=1\otimes \Delta(\Lambda)
\end{gather*}
on $V$. Granovski{\u\i} and Zhedanov \cite{gz93} realized that when restricting to an eigenspace of $\Delta_2(\Lambda)$ on $V$, there exist complex scalars $\varrho$, $\varrho^*$, $\eta$, $\eta^*$, $\omega$ such that the operators $K_0$, $K_1$ and their $q$-commutator
\begin{gather}\label{e:AW1}
K_2=qK_0K_1-q^{-1}K_1K_0
\end{gather}
satisfy the relations 
\begin{eqnarray}
qK_1K_2-q^{-1}K_2K_1
&=&
\omega K_1+\varrho K_0+\eta^*,
\label{e:AW2}\\
qK_2K_0-q^{-1}K_0K_2
&=&
\omega K_0+\varrho^*K_1+\eta.
\label{e:AW3}
\end{eqnarray}

The unital associative algebra over $\C$ generated by three generators, abusively denoted by $K_0$, $K_1$, $K_2$, subject to the relations of the forms (\ref{e:AW1})--(\ref{e:AW3}) is called the {\it Askey--Wilson algebra} 
\cite{hidden_sym}. 
A universal analog $\triangle_q$ of Askey--Wilson algebras was recently proposed by Terwilliger \cite{uaw2011} with $q^4\not=1$. The defining generators of $\triangle_q$ are usually denoted by $A$, $B$, $C$ and the relations assert that each of
\begin{gather*}
A+
\frac{qBC-q^{-1}CB}{q^2-q^{-2}},
\qquad
B+
\frac{qCA-q^{-1}AC}{q^2-q^{-2}},
\qquad
C+
\frac{qAB-q^{-1}BA}{q^2-q^{-2}}
\end{gather*}
is central in $\triangle_q$. Let $\alpha$, $\beta$, $\gamma$ denote the central elements of $\triangle_q$ given by
\begin{eqnarray*}
\frac{\alpha}{q+q^{-1}}
&=&
A+\frac{qBC-q^{-1}CB}{q^2-q^{-2}},
\\
\frac{\beta}{q+q^{-1}}
&=&
B+
\frac{qCA-q^{-1}AC}{q^2-q^{-2}},
\\
\frac{\gamma}{q+q^{-1}}
&=&
C+
\frac{qAB-q^{-1}BA}{q^2-q^{-2}}.
\end{eqnarray*}
The Casimir element of $\triangle_q$ has the expression
\begin{gather*}
\Omega=q ABC+q^2 A^2+q^{-2}B^2+q^2C^2-qA\alpha-q^{-1}B\beta-qC\gamma.
\end{gather*}

Inspired by the work \cite{gz93}, we discover a homomorphism $\flat:\triangle_q\to \U\otimes \U\otimes \U$ that sends
\begin{eqnarray*}
A
&\mapsto&
\Delta(\Lambda)\otimes 1,
\\
B
&\mapsto&
1\otimes \Delta(\Lambda),
\\
\alpha
&\mapsto&
\Lambda\otimes \Lambda\otimes 1
+(1\otimes 1\otimes \Lambda)\cdot\Delta_2(\Lambda),
\\
\beta
&\mapsto&
1\otimes \Lambda\otimes \Lambda
+(\Lambda\otimes 1\otimes 1)\cdot\Delta_2(\Lambda),
\\
\gamma
&\mapsto&
\Lambda\otimes 1\otimes \Lambda
+(1\otimes \Lambda\otimes 1)
\cdot
\Delta_2(\Lambda),
\end{eqnarray*}
and 
\begin{eqnarray*}
C
&\mapsto &
\frac{\Lambda\otimes 1\otimes \Lambda
+(1\otimes \Lambda\otimes 1)\cdot \Delta_2(\Lambda)}
{q+q^{-1}}
\\
&&\;+\;
\frac{
q^{-1}(1\otimes \Delta(\Lambda))\cdot (\Delta(\Lambda)\otimes 1)
-q(\Delta(\Lambda)\otimes 1)\cdot (1\otimes \Delta(\Lambda))}
{q^2-q^{-2}}.
\end{eqnarray*}
The image of $\Omega$ under $\flat$ is equal to
\begin{gather*}
(q+q^{-1})^2
-\Lambda^2\otimes 1\otimes 1
-1\otimes \Lambda^2\otimes 1
-1\otimes 1\otimes \Lambda^2
-(\Lambda\otimes \Lambda\otimes \Lambda)\cdot \Delta_2(\Lambda)
-\Delta_2(\Lambda)^2.
\end{gather*}
Moreover the homomorphism $\flat$ is shown to be injective. As a consequence, given any Hopf $*$-algebra of $\U$ we show that there exists a unique algebra involution $\dag$ of $\triangle_q$ such that the following diagram commutes:
\begin{table}[H]
\centering
\begin{tikzpicture}
\matrix(m)[matrix of math nodes,
row sep=2.6em, column sep=2.8em,
text height=1.5ex, text depth=0.25ex]
{
\triangle_q
&\U\otimes \U\otimes \U\\
\triangle_q
&\U\otimes \U\otimes \U\\
};
\path[->,font=\scriptsize,>=angle 90]
(m-1-1) edge node[auto] {$\flat$} (m-1-2)
(m-1-1) edge node[left] {$\dag$} (m-2-1)
(m-2-1) edge node[auto] {$\flat$} (m-2-2)
(m-1-2) edge node[auto] {$*\otimes *\otimes *$} (m-2-2);
\end{tikzpicture}
\end{table}
\noindent See \S\ref{s:embedd} for details.

Each $\U\otimes \U\otimes \U$-module can be considered as a $\triangle_q$-module via pulling back the injection $\flat$. The aim of \S\ref{s:RW} is to study the $3$-fold tensor products of irreducible Verma $\U$-modules and of finite-dimensional irreducible $\U$-modules from the view of $\triangle_q$-modules. The $\triangle_q$-modules are completely reducible. Furthermore the work  \cite{Huang:2015} allows us to give the decomposition rules for these $\triangle_q$-modules into the direct sums of finite-dimensional irreducible $\triangle_q$-modules. As a consequence, on each irreducible component $A$ and $B$ form a Leonard pair \cite{lp2001,lp&awrelation}. The preliminaries are contained in  \S\ref{s:U} and \S\ref{s:AW}.

It is worthwhile to remark that the paper \cite{Huang:CG} produces a similar work on an algebraic structure behind the Clebsch--Gordan coefficients of $\U$.

\section{The quantum algebra $\U$ and its modules}\label{s:U}

The conventions for this paper are as follows. An algebra is meant to be a unital associative algebra. Let $\Z$ denote the ring of integers. Let $\N$ denote the set of nonnegative integers and $\N^*=\N\setminus\{0\}$.

Let $\F$ denote a field and fix a nonzero scalar $q\in \F$ with $q^4\not=1$. Recall that the $q$-brackets are defined as
\begin{gather*}
[n]=\frac{q^n-q^{-n}}{q-q^{-1}}
\qquad
\hbox{for all $n\in \N$}
\end{gather*}
and the $q$-binomial coefficients are defined as
\begin{gather*}
{n\brack i}
=\prod_{h=1}^i\frac{[n-h+1]}{[h]}
\qquad
\hbox{for all $i,n\in \N$}.
\end{gather*}

The unadorned tensor products $\otimes$ are taken over $\F$.
Given an $\F$-algebra $\A$ the $n$-fold tensor product $A^{\otimes n}$ of $\A$ is recurrently defined by
\begin{gather*}
\A^{\otimes 0}=\F,
\qquad
\A^{\otimes n}=
\A^{\otimes (n-1)}\otimes \A
\qquad
\hbox{for all $n\in \N^*$}.
\end{gather*}
Here $\A^{\otimes 1}$ is identified with $\A$ via the canonical map.
Given an $\F$-algebra homomorphism $\phi$ the $n$-fold tensor product $\phi^{\otimes n}$ of $\phi$ is defined in a similar way: Set $\phi^{\otimes 0}$ to be the identity map on $\F$ and $\phi^{\otimes n}=
\phi^{\otimes (n-1)}\otimes \phi$ for all $n\in \N^*$.

\subsection{The quantum algebra $\U$}
The quantum algebra $\U$ is an $\F$-algebra generated by $E$, $K^{\pm 1}$, $F$ subject to the relations
\begin{gather*}
K K^{-1}=K^{-1}K=1,\\
KE=q^2EK,
\qquad
KF=q^{-2}FK, \\
EF-FE=\frac{K-K^{-1}}{q-q^{-1}}.
\end{gather*}
Recall from \cite[Theorem 1.5]{jantzen} that

\begin{lem}\label{lem:PBW}
The elements
\begin{gather*}
F^i K^h E^j
\qquad
\hbox{
for all
$h\in \Z$
 and
$i, j\in \N$
}
\end{gather*}
are an $\F$-basis of $\U$.
\end{lem}

For each $n\in \Z$ let $U_n$ denote the $\F$-subspace of $\U$ spanned by
\begin{gather*}
F^i K^h E^j
\qquad
\hbox{for all $h\in \Z$ and $i,j\in \N$ with $j-i=n$}.
\end{gather*}
By \cite[\S1.9]{jantzen} the $\F$-subspaces $\{U_n\}_{n\in \Z}$ of $\U$ give a $\Z$-graded structure of $\U$. In other words we have
\setlist[enumerate,1]{leftmargin=2.4em}
\begin{enumerate}[topsep=0.5em,itemsep=0.3em,partopsep=0.5em]
\item[(G1)] $\U=\bigoplus_{n\in \Z} U_n$.

\item[(G2)] $U_m\cdot U_n\subseteq U_{m+n}$ for all $m,n\in \Z$.
\end{enumerate}
\setlist[enumerate,1]{leftmargin=2em}

By \cite[Lemma 2.7]{jantzen} the element
\begin{gather*}
EF + \frac{q^{-1} K+q K^{-1}}{(q-q^{-1})^2}
=
FE + \frac{q K+q^{-1} K^{-1}}{(q-q^{-1})^2}
\end{gather*}
is central in $\U$. The central element is called the {\it Casimir element} of $\U$.
Multiplying the Casimir element by $(q-q^{-1})^2$ we obtain the element
\begin{align}\label{e:Lambda}
\begin{split}
\Lambda
&=
(q-q^{-1})^2 EF + q^{-1} K+q K^{-1}
\\
&=
(q-q^{-1})^2 FE + q K+q^{-1} K^{-1}.
\end{split}
\end{align}
Throughout this paper we use the normalized Casimir element $\Lambda$ instead of the original Casimir element of $\U$.

\begin{lem}\label{lem:EF}
For each $i\in \N$ the following relations {\rm (i)}, {\rm (ii)} hold in $\U$:
\begin{enumerate}
\item
$
E^i F^i =
\displaystyle{
\prod\limits_{h=1}^i
\frac{\Lambda-q^{1-2h}K-q^{2h-1}K^{-1}}{(q-q^{-1})^2}}.$

\item
$
F^i E^i=
\displaystyle{
\prod\limits_{h=1}^i\frac{\Lambda-q^{2h-1}K-q^{1-2h}K^{-1}}{(q-q^{-1})^2}}.$
\end{enumerate}
\end{lem}
\begin{proof}
Proceed by inductions on $i$ and apply (\ref{e:Lambda}).
\end{proof}

\begin{lem}\label{lem:basisU}
For each $n\in \N$ the following {\rm (i)}, {\rm (ii)} hold:
\begin{enumerate}
\item $U_n$ has the $\F$-basis
\begin{gather*}
\Lambda^i K^h E^n
\qquad
h\in \Z,\quad
i\in \N.
\end{gather*}

\item $U_{-n}$ has the $\F$-basis
\begin{gather*}
\Lambda^i K^h F^n
\qquad
h\in \Z,\quad
i\in \N.
\end{gather*}
\end{enumerate}
\end{lem}
\begin{proof}
Fix an $n\in \N$. By Lemma \ref{lem:PBW} the elements
\begin{gather*}
F^i K^h E^{i+n}
\qquad
\hbox{for all $h\in \Z$ and $i\in \N$}
\end{gather*}
are an $\F$-basis of $U_n$. For each $\ell\in \N$ let $U_{n,\ell}$ denote the $\F$-subspace of $U_n$ spanned by $F^i K^h E^{i+n}$ for all $h\in \Z$ and $0\leq i \leq \ell$. Clearly
\begin{gather}\label{e:Un}
U_{n}=\bigcup_{\ell\in \N} U_{n,\ell}.
\end{gather}
We claim that the elements
$$
\Lambda^i K^h E^n
\qquad
\hbox{for all $h\in \Z$ and $0\leq i\leq \ell$}
$$
form an $\F$-basis of $U_{n,\ell}$. To see this we proceed by induction on $\ell$. It is nothing to prove for $\ell=0$. Suppose that $\ell\geq 1$. By construction $U_{n,\ell}/U_{n,\ell-1}$ has the $\F$-basis
\begin{gather*}
F^\ell K^h E^{\ell+n}
+U_{n,\ell-1}
\qquad
\hbox{for all $h\in \Z$}.
\end{gather*}
For each $h\in \Z$ we have
\begin{align*}
F^\ell K^h E^{\ell+n}
=
q^{2h\ell} K^h F^\ell E^{\ell +n} 
=
q^{2h\ell}(q-q^{-1})^{-2\ell}
\Lambda^\ell K^h E^n \pmod{U_{n,\ell-1}}.
\end{align*}
The first equality follows from the defining relation $KF=q^{-2}FK$ and the second equality follows from Lemma \ref{lem:EF}(ii) and the induction hypothesis. Therefore the elements
$$
\Lambda^\ell K^h E^n
+U_{n,\ell-1}
\qquad
\hbox{for all $h\in \Z$}
$$
are an $\F$-basis of $U_{n,\ell} /U_{n,\ell-1}$.  Combined with the induction hypothesis the claim follows. Now (i) follows from the claim and (\ref{e:Un}). Statement (ii) follows by a similar argument.
\end{proof}

Fix $k\in \N^*$. By (G1) and (G2) the $\F$-vector spaces
\begin{gather*}
U_{n_1}\otimes U_{n_2}\otimes \cdots \otimes U_{n_k}
\qquad
\hbox{for all $n_1,n_2,\ldots,n_k\in \Z$}
\end{gather*}
provide a $\Z^k$-graded structure of $\U^{\otimes k}$.
Thus, for each $u\in \U^{\otimes k}$
there exist unique elements $u_{n_1,n_2,\ldots,n_k}\in U_{n_1}\otimes U_{n_2}\otimes \cdots \otimes U_{n_k}$ for all $n_1,n_2,\ldots,n_k\in \Z$ such that
\begin{gather*}
u=\sum_{n_1,n_2,\ldots,n_k\in \Z} u_{n_1,n_2,\ldots,n_k}.
\end{gather*}
Here $u_{n_1,n_2,\ldots,n_k}$ is called the {\it homogeneous component of $u$ of degree $(n_1,n_2,\ldots,n_k)$}.

\begin{lem}\label{lem:basisfoldU}
For all $k\in \N^*$ and $n_1,n_2,\ldots,n_k\in \Z$ 
the elements 
\begin{align*}
\Lambda^{i_1} K^{h_1} X_1^{n_1}
\otimes
\Lambda^{i_2} K^{h_2} X_2^{n_2}
\otimes
\cdots
\otimes
\Lambda^{i_k} K^{h_k} X_k^{n_k}
\qquad
i_1,i_2,\ldots,i_k\in \N,
\quad 
h_1,h_2,\ldots,h_k\in \Z
\end{align*}
are an $\F$-basis of $U_{n_1}\otimes U_{n_2}\otimes \cdots \otimes U_{n_k}$ where $X_i=E$ if $n_i\geq 0$ and $X_i=F$ if $n_i<0$ for all $1\leq i\leq k$.
\end{lem}
\begin{proof}
Immediate from Lemma \ref{lem:basisU}.
\end{proof}

\subsection{The Hopf algebra structure of $\U$}\label{s:hopf}

Recall that the algebras and algebra homomorphisms can be defined in the following ways:
An $\F$-vector space $\A$ endowed with two $\F$-linear maps $\nabla:\A\otimes \A\to \A$ and $\iota:\F\to \A$ is called an {\it $\F$-algebra} if
\begin{gather*}
\nabla\circ (\nabla\otimes 1)=\nabla\circ (1\otimes \nabla)
\end{gather*}
and $\nabla\circ (1\otimes \iota)$ and $\nabla\circ (\iota\otimes 1)$ are equal to the canonical maps $\A\otimes \F\to \A$ and $\F\otimes \A\to \A$, respectively. Here $\nabla$ and $\iota$ are called the {\it multiplication} and the {\it unit} of an $\F$-algebra $\A$ respectively. Suppose that $\A$ and $\A'$ are two $\F$-algebras. Let $\nabla$, $\nabla'$ denote the multiplications of $\A$, $\A'$ and let $\iota$, $\iota'$ denote the units of $\A$, $\A'$ respectively. An $\F$-linear map $\phi:\A\to\A'$ is called an {\it $\F$-algebra homomorphism} if
\begin{gather*}
\nabla'\circ (\phi\otimes \phi)=\phi\circ \nabla,
\\
\iota'=\phi\circ \iota.
\end{gather*}

The coalgebras and coalgebra homomorphisms are defined in dual ways:
An $\F$-vector space $\A$ endowed with two $\F$-linear maps $\Delta:\A\to \A\otimes \A$ and $\epsilon:\A\to \F$ is called an {\it $\F$-coalgebra} if
\begin{gather}\label{e:Delta}
(\Delta\otimes 1)\circ \Delta=(1\otimes \Delta)\circ \Delta
\end{gather}
and $(1\otimes \epsilon)\circ \Delta$ and $(\epsilon\otimes 1)\circ \Delta$ are equal to the canonical maps $\A\to \A\otimes \F$ and $\A\to \F\otimes \A$, respectively.
Here $\Delta$ and $\epsilon$ are called the {\it comultiplication} and the {\it counit} of an $\F$-coalgebra $\A$ respectively. Suppose that $\A$ and $\A'$ are two $\F$-coalgebras. Let $\Delta$, $\Delta'$ denote the comultiplications of $\A$, $\A'$ and let $\epsilon$, $\epsilon'$ denote the counits of $\A$, $\A'$ respectively. An $\F$-linear map $\phi:\A\to\A'$ is called an {\it $\F$-coalgebra homomorphism} if
\begin{gather}
(\phi\otimes \phi)\circ \Delta=\Delta'\circ \phi,
\label{e:coalghom}\\
\epsilon=\epsilon'\circ \phi. 
\notag
\end{gather}
An $\F$-vector space $\A$ is called an {\it $\F$-bialgebra} if $\A$ is an $\F$-algebra and an $\F$-coalgebra such that the multiplication and the unit of $\A$ are $\F$-coalgebra homomorphisms or equivalently the comultiplication and the counit of $\A$ are $\F$-algebra homomorphisms. Suppose that $\A$ and $\A'$ are two $\F$-bialgebras.
An $\F$-linear map $\phi:\A\to \A'$ is called an {\it $\F$-bialgebra homomorphism} if $\phi$ is an $\F$-algebra homomorphism and an $\F$-coalgebra homomorphim. An $\F$-bialgebra $\A$ endowed with an $\F$-linear map $S:\A\to \A$ is called an {\it $\F$-Hopf algebra} if the multiplication $\nabla$, the unit $\iota$, the comultiplication $\Delta$ and the counit $\epsilon$ of $\A$ satisfy 
\begin{gather*}
\nabla \circ (1\otimes S)\circ \Delta
=\iota\circ\epsilon
=\nabla \circ (S\otimes 1)\circ \Delta.
\end{gather*}
Here $S$ is called the {\it antipode} of an $\F$-Hopf algebra $\A$. 
Suppose that $\A$ and $\A'$ are two $\F$-Hopf algebras. Let $S$, $S'$ denote the antipodes of $\A$, $\A'$ respectively.
An $\F$-linear map $\phi:\A\to \A'$ is called an {\it $\F$-Hopf algebra homomorphism} if $\phi$ is an $\F$-bialgebra homomorphism satisfying 
\begin{gather*}
\phi\circ S=S'\circ \phi.
\end{gather*}

The quantum algebra $\U$ admits the Hopf algebra structure: The comultiplication $\Delta:\U\to \U\otimes \U$ is an $\F$-algebra homomorphism defined by
\begin{eqnarray}
\Delta(E) &=& E\otimes 1+K\otimes E,
\label{e:DE}
\\
\Delta(F) &=& F\otimes K^{-1}+1\otimes F,
\label{e:DF}
\\
\Delta(K) &=& K\otimes K.
\label{e:DK}
\end{eqnarray}
The counit $\epsilon:\U\to \F$ is an $\F$-algebra homomorphism defined by
\begin{gather*}
\epsilon (E)=0,
\qquad
\epsilon (F)=0,
\qquad
\epsilon (K)=1.
\end{gather*}
The antipode $S:\U\to \U$ is an $\F$-algebra antiautomorphism defined by
\begin{gather*}
S(E)=-K^{-1} E,
\qquad
S(F)=-FK,
\qquad
S(K)=K^{-1}.
\end{gather*}
Let $\Delta_0$ denote the identity map on $\U$.
We recurrently define
\begin{gather*}
\Delta_{n}=(
1^{\otimes (n-1)}
\otimes \Delta)\circ \Delta_{n-1}
\qquad
\hbox{for all $n\in \N^*$}.
\end{gather*}
The map $\Delta_n$ ($n\in \N$) is called the {\it $n$-fold comultiplication} of $\U$.

\begin{lem}\label{lem:Dn}
For all $n\in \N^*$ and $1\leq i\leq n$ the following equation holds:
$$
\Delta_{n}=(
1^{\otimes n-i}
\otimes
\Delta
\otimes
1^{\otimes  i-1}
)
\circ \Delta_{n-1}.
$$
\end{lem}
\begin{proof}
Proceed by induction on $n$ and apply (\ref{e:Delta}).
\end{proof}

Note that each $\U^{\otimes n}$-module can be regarded as a $\U$-module by pulling back via $\Delta_{n-1}$ for $n\in \N^*$.

\subsection{Verma $\U$-modules and finite-dimensional $\U$-modules}\label{s:Umodule}

For notational convenience we set
$$
[\lambda;n]=\frac{\lambda q^n-\lambda^{-1} q^{-n}}{q-q^{-1}}
\qquad
\hbox{for any nonzero $\lambda\in \F$ and $n\in \Z$}.
$$

For any nonzero scalar $\lambda\in \F$  there exists a unique $\U$-module
$M(\lambda)$ that has an $\F$-basis $\{m^{(\lambda)}_i\}_{i\in \N}$ such that
\begin{eqnarray*}
K m^{(\lambda)}_i
&=&
\lambda q^{-2i} m^{(\lambda)}_i
\qquad
\hbox{for all $i\in \N$},
\\
F m^{(\lambda)}_i
&=&
[i+1]m^{(\lambda)}_{i+1}
\qquad
\hbox{for all $i\in \N$},
\\
Em^{(\lambda)}_i
&=&
[\lambda;1-i]
m^{(\lambda)}_{i-1}
\qquad
\hbox{for all $i\in \N^*$},
\qquad
E m^{(\lambda)}_0=0.
\end{eqnarray*}
The $\U$-module $M(\lambda)$ is called the {\it Verma $\U$-module}. By \cite[\S2.4]{jantzen} the Verma $\U$-module $M(\lambda)$ satisfies the following universal property:

\begin{prop}\label{prop:UPUqsl2}
Let $\lambda$ denote a nonzero scalar in $\F$. If a $\U$-module $V$ contains a vector $v$ such that
\begin{gather*}
Ev=0,
\qquad
Kv=\lambda v,
\end{gather*}
then there exists a unique $\U$-module homomorphism $M(\lambda)\to V$ that sends $m^{(\lambda)}_0$ to $v$.
\end{prop}

Recall the normalized Casimir element $\Lambda$ from (\ref{e:Lambda}). A direct calculation yields that

\begin{lem}\label{lem:CasVerma}
Let $\lambda$ denote a nonzero scalar in $\F$. Then the element $\Lambda$ acts on the Verma $\U$-module $M(\lambda)$ as scalar multiplication by
$
\lambda q+\lambda^{-1} q^{-1}.
$
\end{lem}

A necessary and sufficient condition for the Verma $\U$-module $M(\lambda)$ to be irreducible is given below. For a proof please see \cite[Proposition 2.5]{jantzen}.

\begin{lem}\label{lem:irrVerma}
For any nonzero scalar $\lambda\in \F$ the following {\rm (i)}, {\rm (ii)} are equivalent:
\begin{enumerate}
\item The Verma $\U$-module $M(\lambda)$ is irreducible.

\item 
$\lambda\not\in \{\pm q^n\,|\,n\in \N\}$. 
\end{enumerate}
\end{lem}

Fix $n\in \N$ and $\e\in \{\pm 1\}$. Now we set $\lambda=\e q^n$.
Let $N(\lambda)$ denote the $\F$-subspace of $M(\lambda)$ spanned by
\begin{gather*}
m^{(\lambda)}_i
\qquad
\hbox{for all $i\geq n+1$}.
\end{gather*}
Clearly $N(\lambda)$ is invariant under $K$ and $F$.
Since $[\lambda; -n]=0$ in the setting, the $\F$-vector space $N(\lambda)$ is invariant under $E$.
Therefore $N(\lambda)$ is a $\U$-submodule of $M(\lambda)$. Define
\begin{gather*}
V(n,\e)=M(\lambda)/N(\lambda).
\end{gather*}
By construction the $\U$-module $V(n,\e)$ has the $\F$-basis
\begin{gather*}
v^{(n,\e)}_i=m^{(\lambda)}_i+N(\lambda)
\qquad
(0\leq i\leq n).
\end{gather*}
The action of $K$, $F$, $E$ on $V(n,\e)$ with respect to the $\F$-basis $\{v_i^{(n,\e)}\}_{i=0}^n$ of $V(n,\e)$ is as follows:
\begin{eqnarray*}
K v_i^{(n,\e)} &=& \e q^{n-2i} v_i^{(n,\e)} \qquad  (0\leq i\leq n),
\\
F v_i^{(n,\e)} &=& [i+1] v_{i+1}^{(n,\e)} \qquad  (0\leq i\leq n-1),
\qquad
F v_n^{(n,\e)} = 0,
\\
E v_i^{(n,\e)} &=& \e [n-i+1] v_{i-1}^{(n,\e)} \qquad  (1\leq i\leq n),
\qquad
E v_0^{(n,\e)} =0.
\end{eqnarray*}

\begin{lem}\label{lem:Casimir}
Let $n\in \N$ and $\e\in \{\pm 1\}$. Then the element $\Lambda$ acts on the $\U$-module $V(n,\e)$ as scalar multiplication by
$\e(q^{n+1}+q^{-n-1})$.
\end{lem}
\begin{proof}
Immediate from Lemma \ref{lem:CasVerma}.
\end{proof}

A necessary and sufficient condition for $V(n,\e)$ to be irreducible is given below.

\begin{lem}\label{lem:irrVn}
For all $n\in \N$ and $\e\in \{\pm 1\}$ the following {\rm (i)}, {\rm (ii)} are equivalent:
\begin{enumerate}
\item The $\U$-module $V(n,\e)$ is irreducible.

\item $q^{2i}\not=1$ for all $1\leq i\leq n$.
\end{enumerate}
\end{lem}
\begin{proof}
(i) $\Rightarrow$ (ii): Suppose on the contrary that there exists $1\leq i\leq n$ such that $q^{2i}=1$. Clearly the $\F$-subspace $V$ of $V(n,\e)$ spanned by
$$
v^{(n,\e)}_k
\qquad
\hbox{for all $0\leq k\leq i-1$}
$$
is invariant under $E$ and $K$. Since $[i]=0$ the $\F$-vector space $V$ is invariant under $F$. Therefore $V$ is a nonzero $\U$-submodule of $V(n,\e)$, a contradiction to (i).

(ii) $\Rightarrow$ (i): To see the irreducibility of $V(n,\e)$, we let $V$ denote any nonzero $\U$-submodule of $V(n,\e)$ and show that $V=V(n,\e)$. Choose a nonzero vector $v\in V$. For $0\leq i\leq n$ let $c_i$ denote the coefficient of $v^{(n,\e)}_i$ in $v$ with respect to the $\F$-basis $\{v_i^{(n,\e)}\}_{i=0}^n$ of $V(n,\e)$. Choose $k=\max\{i\,|\,c_i\not=0\}$. Observe that
$$
E^{k} v
=
\e^k c_k
[n][n-1]\cdots [n-k+1] 
v_0^{(n,\e)}.
$$
Since $V$ is a $\U$-module the vector $E^k v\in V$.
By (ii) the scalars $[i]\not=0$ for all $1\leq i\leq n$. Hence $v_0^{(n,\e)}\in V$. Observe that 
$$
F^i v_0^{(n,\e)}=
[1][2]\cdots [i]v_i^{(n,\e)}
\qquad 
(1\leq i\leq n).
$$
Hence $v_i^{(n,\e)}\in V$ for all $1\leq i\leq n$. Therefore $V=V(n,\e)$.
\end{proof}

For all $n\in \N$ it is known that any $(n+1)$-dimensional irreducible $\U$-module is isomorphic to $V(n,1)$ or $V(n,-1)$ provided that $\F$ is algebraically closed and $q$ is not a root of unity \cite[Theorem 2.6]{jantzen}.

\section{The universal Askey-Wilson algebra $\triangle_q$ and its modules}\label{s:AW}

\subsection{The universal Askey-Wilson algebra}

The {\it universal Askey--Wilson algebra} $\triangle_q$ is an $\F$-algebra generated by $A$, $B$, $C$ and the relations assert that each of
\begin{gather*}
A+
\frac{qBC-q^{-1}CB}{q^2-q^{-2}},
\qquad
B+
\frac{qCA-q^{-1}AC}{q^2-q^{-2}},
\qquad
C+
\frac{qAB-q^{-1}BA}{q^2-q^{-2}}
\end{gather*}
commutes with $A$, $B$, $C$ \cite[Definition 1.2]{uaw2011}. Recall from \S\ref{s:intro} that the {\it Askey-Wilson algebra} is an $\F$-algebra generated by $K_0$, $K_1$, $K_2$ subject to the relations of the forms (\ref{e:AW1})--(\ref{e:AW3}). Recall that (\ref{e:AW1})--(\ref{e:AW3}) involve five additional parameters $\varrho$, $\varrho^*$, $\eta$, $\eta^*$, $\omega$.  It is straightforward to verify that $\triangle_q$ satisfies the following property: Under the mild constraints $\varrho\not=0$ and $\varrho^*\not=0$ there exists a unique surjective homomorphism from $\triangle_q$ into the Askey-Wilson algebra that sends
\begin{gather*}
A\mapsto -\frac{q^2-q^{-2}}{\sqrt{\varrho^*}} K_0,\qquad
B\mapsto  -\frac{q^2-q^{-2}}{\sqrt \varrho} K_1, \qquad
C\mapsto \frac{q+q^{-1}}{\sqrt{\varrho\varrho^*}}\omega
-\frac{q^2-q^{-2}}{\sqrt{\varrho\varrho^*}}K_2.
\end{gather*}

Let $\alpha$, $\beta$, $\gamma$ denote the elements of $\triangle_q$ defined by
\begin{eqnarray*}
\frac{\alpha}{q+q^{-1}}
&=&
A+\frac{qBC-q^{-1}CB}{q^2-q^{-2}},
\\
\frac{\beta}{q+q^{-1}}
&=&
B+
\frac{qCA-q^{-1}AC}{q^2-q^{-2}},
\\
\frac{\gamma}{q+q^{-1}}
&=&
C+
\frac{qAB-q^{-1}BA}{q^2-q^{-2}}.
\end{eqnarray*}
By the definition of $\triangle_q$ the elements $\alpha$, $\beta$, $\gamma$ are central in $\triangle_q$. Recall from \cite[Theorem~4.1]{uaw2011} that

\begin{lem}\label{lem:PBWaw}
The elements
\begin{gather*}
A^{i} B^{j} C^{k}
\alpha^{r}
\beta^{s}
\gamma^{t}
\qquad
\hbox{for all $i,j,k,r,s,t\in \N$}
\end{gather*}
are an $\F$-basis of $\triangle_q$.
\end{lem}

Observe that $C$ is an $\F$-linear combination of $AB$, $BA$, $\gamma$. Therefore $A$, $B$, $\gamma$ form a set of generators of $\triangle_q$.  A presentation for $\triangle_q$ with respect to the generators is given in \cite[Theorem~2.2]{uaw2011}:

\begin{lem}\label{lem:ABc}
The $\F$-algebra $\triangle_q$ is generated by  $A$, $B$, $\gamma$ subject to the relations
\begin{gather*}
A^3B-[3]A^2BA+[3]ABA^2-BA^3
=
(q^2-q^{-2})^2(BA-AB),
\\
B^3A-[3]B^2AB+[3]BAB^2-AB^3
=
(q^2-q^{-2})^2(AB-BA),
\\
B^2A^2-(q^2+q^{-2})BABA
+(q^2+q^{-2})ABAB-A^2B^2
=
(q-q^{-1})^2(AB-BA)\gamma,
\\
\gamma A=A\gamma,
\qquad
\gamma B=B\gamma.
\end{gather*}
\end{lem}

By \cite[Theorem 6.2]{uaw2011} the element
\begin{gather}\label{e:Casimirtriangle}
\Omega=q ABC+q^2 A^2+q^{-2}B^2+q^2C^2-qA\alpha-q^{-1}B\beta-qC\gamma
\end{gather}
is central in $\triangle_q$. The central element is called the {\it Casimir element} of $\triangle_q$ \cite{hidden_sym,uaw2011}.

\subsection{Verma $\triangle_q$-modules and finite-dimensional $\triangle_q$-modules}\label{s:AWmodule}
Throughout this subsection, let $a$, $b$, $c$, $\nu$ denote any four nonzero scalars in $\F$. Set
\begin{eqnarray*}
\theta_i
&=&
aq^{2i} \nu^{-1}+a^{-1}q^{-2i} \nu
\qquad
\hbox{for all $i\in \Z$},
\\
\theta_i^*
&=&
bq^{2i} \nu^{-1}+b^{-1}q^{-2i} \nu
\qquad
\hbox{for all $i\in \Z$},
\\
\varphi_i
&=&
a^{-1} b^{-1} q \nu
(q^i-q^{-i})
(q^{i-1} \nu^{-1}- q^{1-i} \nu)
\\
&&\quad \times \; (q^{-i}-abc q^{i-1}  \nu^{-1})
(q^{-i}-abc^{-1} q^{i-1} \nu^{-1} )
\qquad
\hbox{for all $i\in \Z$},
\\
\omega
&=&
(b+b^{-1})
(c+c^{-1})
+
(a+a^{-1})
(q\nu + q^{-1}\nu^{-1}),
\\
\omega^*
&=&
(c+c^{-1})
(a+a^{-1})
+
(b+b^{-1})
(q\nu + q^{-1}\nu^{-1}),
\\
\omega^{\varepsilon}
&=&
(a+a^{-1})
(b+b^{-1})
+
(c+c^{-1})
(q\nu + q^{-1}\nu^{-1}).
\end{eqnarray*}

Recall from \cite[\S3]{Huang:2015} that there exists a unique $\triangle_q$-module $M_\nu(a,b,c)$ with an $\F$-basis $\{m_i\}_{i\in \N}$ such that
\begin{align*}
(A-\theta_i) m_i&=
m_{i+1}
\qquad
\hbox{for all $i\in \N$},
\\
(B-\theta_i^*) m_i&=
\varphi_i m_{i-1}
\qquad
\hbox{for all $i\in \N^*$},
\qquad
(B-\theta_0^*) m_0=0,
\end{align*}
and
\begin{align*}
\left(
C
-\frac{\omega^\e-\theta_i\theta_i^*}{q+q^{-1}}
\right)
m_i&=
\varphi_i \frac{q^{-1}\theta_i-q\theta_{i-1}}
{q^2-q^{-2}}
m_{i-1}
+
\frac{q^{-1}\varphi_{i+1}-q\varphi_i}{q^2-q^{-2}}
m_i
\\
&\qquad+\,
\frac{q^{-1}\theta_{i+1}^*-q\theta_i^*}
{q^2-q^{-2}}m_{i+1}
\qquad
\hbox{for all $i\in \N^*$},
\\
\left(
C
-\frac{\omega^\e-\theta_0\theta_0^*}{q+q^{-1}}
\right)
m_0&=
\frac{q^{-1}\varphi_1}{q^2-q^{-2}}
m_0
+
\frac{q^{-1}\theta_1^*-q\theta_0^*}
{q^2-q^{-2}}m_1.
\end{align*}
The central elements $\alpha$, $\beta$, $\gamma$ act on $M_\nu(a,b,c)$ as scalar multiplications by $\omega$, $\omega^*$, $\omega^\e$ respectively. The $\triangle_q$-module $M_\nu(a,b,c)$ is called the {\it Verma $\triangle_q$-module}. By \cite[Proposition 3.2]{Huang:2015} the Verma $\triangle_q$-module $M_\nu(a,b,c)$ satisfies the following universal property:

\begin{prop}\label{prop:UP}
If a $\triangle_q$-module $V$ contains a vector $v$ such that
\begin{gather*}
B v=\theta_0^* v,
\\
(B-\theta_1^*)(A-\theta_0) v=\varphi_1 v,
\\
\alpha v=\omega v,
\qquad
\beta v=\omega^*v,
\qquad
\gamma v=\omega^\e v,
\end{gather*}
then there exists a unique $\triangle_q$-module homomorphism $M_\nu(a,b,c)\to V$ that sends $m_0$ to $v$.
\end{prop}

Fix an integer $d\geq 0$ and set $\nu=q^d$. Let $N_\nu(a,b,c)$ denote the $\F$-subspace of $M_\nu(a,b,c)$ spanned by
\begin{gather*}
m_i
\qquad
\hbox{for all $i\geq d+1$}. 
\end{gather*}
Clearly $N_\nu(a,b,c)$ is invariant under $A$. Since $\varphi_{d+1}=0$ in the setting, the $\F$-vector space $N_\nu(a,b,c)$ is invariant under $B$ and $C$. Therefore $N_\nu(a,b,c)$
is a $\triangle_q$-submodule of $M_\nu(a,b,c)$. Define
$$
V_d(a,b,c)=M_\nu(a,b,c)/N_\nu(a,b,c).
$$
By construction the $\triangle_q$-module $V_d(a,b,c)$ has the $\F$-basis
$$
v_i=m_i+N_\nu(a,b,c)
\qquad
(0\leq i\leq d).
$$
The action of $A$, $B$, $C$ on $V_d(a,b,c)$ with respect to the $\F$-basis $\{v_i\}^d_{i=0}$ is as follows:

\begin{align*}
(A-\theta_i) v_i&=
v_{i+1}
\qquad
(0\leq i\leq d-1),
\qquad
(A-\theta_d) v_d= 0,
\\
(B-\theta_i^*) v_i&=
\varphi_i v_{i-1}
\qquad
(1\leq i\leq d),
\qquad
(B-\theta_0^*) v_0=0,
\end{align*}
and
\begin{align*}
\left(
C
-\frac{\omega^\e-\theta_i\theta_i^*}{q+q^{-1}}
\right)
v_i
&=
\varphi_i \frac{q^{-1}\theta_i-q\theta_{i-1}}
{q^2-q^{-2}}
v_{i-1}
+
\frac{q^{-1}\varphi_{i+1}-q\varphi_i}{q^2-q^{-2}}
v_i
\\
&\qquad+\,
\frac{q^{-1}\theta_{i+1}^*-q\theta_i^*}
{q^2-q^{-2}}v_{i+1}
\qquad (1\leq i\leq d-1),
\\
\left(
C
-\frac{\omega^\e-\theta_0\theta_0^*}{q+q^{-1}}
\right)
v_0
&=
\frac{q^{-1}\varphi_1}{q^2-q^{-2}}
v_0
+
\frac{q^{-1}\theta_1^*-q\theta_0^*}
{q^2-q^{-2}}v_1,
\\
\left(
C
-\frac{\omega^\e-\theta_d\theta_d^*}{q+q^{-1}}
\right)v_d
&=
\varphi_d \frac{q^{-1}\theta_d-q\theta_{d-1}}
{q^2-q^{-2}}
v_{d-1}
-
\frac{q\varphi_d}{q^2-q^{-2}}
v_d.
\end{align*}
The central elements $\alpha$, $\beta$, $\gamma$ act on $V_d(a,b,c)$ as scalar multiplications by $\omega$, $\omega^*$, $\omega^\e$ respectively. The $\triangle_q$-module $V_d(a,b,c)$ satisfies the following universal property:

\begin{prop}\label{prop:hom}
If a $\triangle_q$-module $V$ contains a vector $v$ such that
\begin{gather*}
Bv=\theta_0^* v,
\\
(B-\theta_1^*)(A-\theta_0)v=\varphi_1v,
\\
\alpha v=\omega v,
\qquad
\beta v=\omega^* v,
\qquad
\gamma v=\omega^\e v,
\end{gather*}
and the characteristic polynomial of $A$ on $V$ is $\prod\limits_{i=0}^d (X-\theta_i)$, 
then there exists a unique $\triangle_q$-module homomorphism $V_d(a,b,c)\to V$ that sends $v_0$ to $v$.
\end{prop}
\begin{proof}
By Proposition \ref{prop:UP} there exists a unique $\triangle_q$-module $\phi:M_\nu(a,b,c)\to V$ that sends $m_0$ to $v$. 
Observe that
$$
N_\nu(a,b,c)=\displaystyle{\prod_{i=0}^d(A-\theta_i) M_\nu(a,b,c)}.
$$
By the Cayley-Hamilton theorem $N_\nu(a,b,c)$ is contained in the kernel of $\phi$. The lemma follows.
\end{proof}

A necessary and sufficient condition for $V_d(a,b,c)$ to be irreducible is given below. For a proof please see \cite[Theorem~4.4]{Huang:2015}.

\begin{lem}\label{lem:irr} 
The $\triangle_q$-module $V_d(a,b,c)$ is irreducible if and only if the following {\rm (i)}, {\rm (ii)} hold:
\begin{enumerate}
\item $q^{2i}\not=1$ for all $1\leq i\leq d$.

\item $abc,$ $a^{-1}bc,$ $ab^{-1}c,$ $abc^{-1}\not\in \{q^{d+1-2i}\,|\,1\leq i\leq d\}$.
\end{enumerate}
\end{lem}

By \cite[Theorem~4.7]{Huang:2015}, for any $(d+1)$-dimensional irreducible $\triangle_q$-module $V$, there exist nonzero scalars $a,b,c\in \F$ satisfying Lemma \ref{lem:irr}(ii) such that $V$ is isomorphic to $V_d(a,b,c)$, provided that $\F$ is algebraically closed and $q$ is not a root of unity.

\subsection{Finite-dimensional irreducible $\triangle_q$-modules and Leonard pairs}

In this subsection we give a brief introduction to two famous families of finite-dimensional irreducible $\triangle_q$-modules studied in \cite{Huang:2012,Huang:2015}.

To describe these $\triangle_q$-modules, we begin with some terms. A matrix is said to be {\it tridiagonal} if all nonzero entries lie on either the diagonal, the subdiagonal, or the superdiagonal. A tridiagonal matrix is said to be {\it irreducible} if each entry on the subdiagonal and superdiagonal is nonzero. 
Recall from \cite{lp2001} that two linear transformations on a finite-dimensional vector space $V$ are called a {\it Leonard pair} whenever for each of the two transformations, there exists a basis of $V$ with respect to which the matrix representing that transformation is diagonal and the matrix representing the other transformation is irreducible tridiagonal. \cite[Theorem 5.2]{Huang:2015} gave the following necessary and sufficient conditions for any two of $A$, $B$, $C$ as a Leonard pair on a finite-dimensional irreducible $\triangle_q$-module.

\begin{lem}\label{lem:lp}
Assume that the $\triangle_q$-module $V_d(a,b,c)$ is irreducible. Then the following {\rm (i)}--{\rm(iii)} are equivalent:
\begin{enumerate}
\item $A$, $B$ {\rm (}resp. $B$, $C${\rm )}  {\rm (}resp. $C$, $A${\rm )} act on $V_d(a,b,c)$ as a Leonard pair.

\item $A$, $B$  {\rm (}resp. $B$, $C${\rm )}  {\rm (}resp. $C$, $A${\rm )} are diagonalizable on $V_d(a,b,c)$.

\item Neither of $a^2$, $b^2$  {\rm (}resp. $b^2$, $c^2${\rm )}  {\rm (}resp. $c^2$, $a^2${\rm )} is in the set $\{q^{2(d-i)}\,|\,1\leq i\leq 2d-1\}$.
\end{enumerate}
\end{lem}

Recall from \cite{cur2007} that three linear transformations on a finite-dimensional vector space $V$ are called a {\it Leonard triple} \cite{cur2007} whenever for each of the three transformations, there exists a basis of $V$ with respect to which the matrix representing that transformation is diagonal and the matrices representing the other transformations are irreducible tridiagonal. \cite[Theorem 5.3]{Huang:2015} gave the following necessary and sufficient conditions for $A$, $B$, $C$ as a Leonard triple on a finite-dimensional irreducible $\triangle_q$-module.

\begin{lem}\label{lem:lt}
Assume that the $\triangle_q$-module $V_d(a,b,c)$ is irreducible. Then the following {\rm (i)}--{\rm(iv)} are equivalent:
\begin{enumerate}
\item $A$, $B$, $C$ act on $V_d(a,b,c)$ as a Leonard triple.

\item Any two of $A$, $B$, $C$ act on $V_d(a,b,c)$ as a Leonard pair.

\item $A$, $B$, $C$ are diagonalizable on $V_d(a,b,c)$.

\item $a^2,$ $b^2,$ $c^2\not\in\{q^{2(d-i)}\,|\,1\leq i\leq 2d-1\}$.
\end{enumerate}
\end{lem}

\section{An embedding of $\triangle_q$ into $\U\otimes \U\otimes \U$}\label{s:embedd}
The outline of this section is as follows. In \S\ref{s:homo} we prove the existence and uniqueness of the homomorphism $\flat:\triangle_q\to \U\otimes \U\otimes \U$ described in \S\ref{s:intro}. Moreover we display the image of the Casimir element $\Omega$ of $\triangle_q$ under $\flat$ in terms of the normalized Casimir element $\Lambda$ of $\U$. In \S\ref{s:inj} we prove that the homomorphism $\flat$ is injective. As an application of the injectivity of $\flat$, in \S\ref{s:hopf&AW} we show that 
for any Hopf $*$-algebra of $\U$, the involution $*\otimes *\otimes *$ of $\U\otimes \U\otimes \U$ can be uniquely pulled back to an algebra involution of $\triangle_q$ via $\flat$.

\subsection{A homomorphism $\flat$ of $\triangle_q$ into $\U\otimes \U\otimes \U$}\label{s:homo}

In this subsection we prove the principal result of this paper:

\begin{thm}\label{thm:embedd}
There exists a unique $\F$-algebra homomorphism $\flat:\triangle_q \to \U\otimes \U\otimes \U$ such that
\begin{eqnarray*}
A^\flat
&=&
\Delta(\Lambda)\otimes 1,
\label{natural:A}
\\
B^\flat
&=&
1\otimes \Delta(\Lambda),
\label{natural:B}
\\
\alpha^\flat
&=&
\Lambda\otimes \Lambda\otimes 1
+(1\otimes 1\otimes \Lambda)\cdot\Delta_2(\Lambda),
\label{natural:alpha}
\\
\beta^\flat
&=&
1\otimes \Lambda\otimes \Lambda
+(\Lambda\otimes 1\otimes 1)\cdot\Delta_2(\Lambda),
\label{natural:beta}
\\
\gamma^\flat
&=&
\Lambda\otimes 1\otimes \Lambda
+(1\otimes \Lambda\otimes 1)
\cdot
\Delta_2(\Lambda),
\label{natural:gamma}
\end{eqnarray*}
and 
\begin{eqnarray*}
C^\flat
&=&
\frac{\Lambda\otimes 1\otimes \Lambda
+(1\otimes \Lambda\otimes 1)\cdot \Delta_2(\Lambda)}
{q+q^{-1}}
\\
&&\;+\;
\frac{
q^{-1}(1\otimes \Delta(\Lambda))\cdot (\Delta(\Lambda)\otimes 1)
-q(\Delta(\Lambda)\otimes 1)\cdot (1\otimes \Delta(\Lambda))}
{q^2-q^{-2}}.
\end{eqnarray*}
\end{thm}

To verify the existence of $\flat$ we make some preparation.

\begin{lem}\label{lem:commu}
For all $n\in \N^*$
the elements $1\otimes \Delta_{n-1}(\Lambda)$ and $\Delta_{n-1}(\Lambda)\otimes 1$ are contained in the centralizer of $\Delta_{n}(\U)$ in $\U^{\otimes (n+1)}$.
\end{lem}
\begin{proof}
Proceed by induction on $n$. Since $1$, $\Lambda$ are central in $\U$
each of $1\otimes \Lambda$, $\Lambda\otimes 1$ commutes with all elements of $\Delta(\U)$. Therefore the lemma holds for $n=1$. Now suppose $n\geq 2$. Let $u$ denote any element of $\U$.
By induction hypothesis the element $\Delta_{n-1}(u)$ commutes with
$1\otimes \Delta_{n-2}(\Lambda)$ and $\Delta_{n-2}(\Lambda)\otimes 1$.
The elements $\Delta_{n}(u)$ and $1\otimes \Delta_{n-1}(\Lambda)$ are the images of $\Delta_{n-1}(u)$ and $1\otimes \Delta_{n-2}(\Lambda)$ under the homomorphism
$1^{\otimes(n-1)}\otimes \Delta$,
respectively. Hence $\Delta_{n}(u)$ commutes with $1\otimes \Delta_{n-1}(\Lambda)$. By Lemma \ref{lem:Dn} the elements $\Delta_{n}(u)$ and $\Delta_{n-1}(\Lambda)\otimes 1$ are the images of $\Delta_{n-1}(u)$ and $\Delta_{n-2}(\Lambda)\otimes 1$ under the homomorphism
$\Delta\otimes 1^{\otimes(n-1)}$,
respectively. Hence $\Delta_{n}(u)$ commutes with $\Delta_{n-1}(\Lambda)\otimes 1$. The lemma follows.
\end{proof}

\begin{lem}\label{lem:DeltaL}
The homogeneous components of $\Delta(\Lambda)$ and $\Delta_2(\Lambda)$ are given in the tables below. Any homogeneous component not displayed is zero.

\begin{table}[H]
\small
\centering
\extrarowheight=3pt
\begin{tabular}{c|c}
degree  &$\Delta(\Lambda)$\\

\midrule[2pt]

$(-1,1)$
&$
q^2(q-q^{-1})^2 KF\otimes K^{-1}E
$\\

\midrule[1pt]

\multirow{2}{*}{$(0,0)$}
&$
\Lambda\otimes K^{-1}
+K\otimes \Lambda
$\\

&$
\;-\;(q+q^{-1}) K\otimes K^{-1}
$\\

\midrule[1pt]

$(1,-1)$
&$
(q-q^{-1})^2 E\otimes F
$
\end{tabular}
\end{table}


\begin{table}[H]
\small
\centering
\extrarowheight=3pt
\begin{tabular}{c|c}
degree &$\Delta_2(\Lambda)$\\

\midrule[2pt]

$(-1,0,1)$
&$
q^2(q-q^{-1})^2 KF\otimes 1\otimes K^{-1}E
$\\

\midrule[1pt]

$(-1,1,0)$
&$
q^2(q-q^{-1})^2 KF\otimes K^{-1}E\otimes K^{-1}
$\\

\midrule[1pt]

$(0,-1,1)$
&$
q^2(q-q^{-1})^2 K\otimes KF\otimes K^{-1}E
$
\\

\midrule[1pt]

\multirow{2}{*}{
$(0,0,0)$
}
&$
K\otimes K\otimes \Lambda
+K\otimes \Lambda\otimes K^{-1}
+\Lambda\otimes K^{-1}\otimes K^{-1}
$
\\

&$
\;-\;
(q+q^{-1})K\otimes K\otimes K^{-1}
-(q+q^{-1})K\otimes K^{-1}\otimes K^{-1}
$\\

\midrule[1pt]

$(0,1,-1)$
&$
(q-q^{-1})^2 K\otimes E\otimes F
$\\

\midrule[1pt]

$(1,-1,0)$
&$
(q-q^{-1})^2 E\otimes F\otimes K^{-1}
$\\

\midrule[1pt]

$(1,0,-1)$
&$
(q-q^{-1})^2 E\otimes 1\otimes F
$
\end{tabular}
\end{table}
\end{lem}
\begin{proof}
By (\ref{e:Lambda}) the element $\Delta(\Lambda)$ is equal to
\begin{gather}\label{e:DL}
(q-q^{-1})^2 \Delta(F)\Delta(E)
+q \Delta(K)
+q^{-1}\Delta(K)^{-1}.
\end{gather}
To get the table for $\Delta(\Lambda)$ one may substitute (\ref{e:DE})--(\ref{e:DK}) into (\ref{e:DL}) and use Lemma \ref{lem:EF}(ii) to evaluate the  homogeneous component of $\Delta(\Lambda)$ of degree $(0,0)$.
By the definition of $\Delta_2$ the element $\Delta_2(\Lambda)=(1\otimes \Delta)(\Delta(\Lambda))$. To get the table of $\Delta_2(\Lambda)$ one may apply the identity to evaluate $\Delta_2(\Lambda)$.
\end{proof}

For convenience, let $A^\flat$, $B^\flat$, $C^\flat$, $\alpha^\flat$, $\beta^\flat$, $\gamma^\flat$ denote the elements of $\U\otimes \U\otimes \U$ as in the statement of Theorem \ref{thm:embedd}.

\begin{lem}\label{lem:ABCflat}
The homogeneous components of  $A^\flat$, $B^\flat$, $C^\flat$, $\alpha^\flat$, $\beta^\flat$, $\gamma^\flat$ are given in the tables below. Any homogeneous component not displayed is zero.

\begin{table}[H]
\begin{minipage}[H]{0.5\textwidth}
\small
\centering
\extrarowheight=3.3pt
\begin{tabular}{c|c}
degree  &$A^\flat$\\

\midrule[2pt]

$(-1,1,0)$
&$
q^2(q-q^{-1})^2 KF\otimes K^{-1}E\otimes 1
$
\\

\midrule[1pt]

\multirow{2}{*}{$(0,0,0)$}
&$
\Lambda\otimes K^{-1}\otimes 1
+K\otimes \Lambda\otimes 1
$
\\

&$
\;-\;(q+q^{-1})K\otimes K^{-1}\otimes 1
$
\\


\midrule[1pt]

$(1,-1,0)$
&$
(q-q^{-1})^2 E\otimes F\otimes 1
$

\end{tabular}
\end{minipage}%
\begin{minipage}[H]{0.5\textwidth}
\small
\centering
\extrarowheight=3pt
\begin{tabular}{c|c}
degree  &$B^\flat$\\

\midrule[2pt]

$(0,-1,1)$
&$
q^2(q-q^{-1})^2\otimes KF\otimes K^{-1}E
$
\\

\midrule[1pt]

\multirow{2}{*}{$(0,0,0)$}
&$
1\otimes \Lambda\otimes K^{-1}
+1\otimes K\otimes \Lambda
$
\\

&$
\;-\;(q+q^{-1})\otimes K\otimes K^{-1}
$
\\


\midrule[1pt]

$(0,1,-1)$
&$
(q-q^{-1})^2\otimes E\otimes F
$
\end{tabular}
\end{minipage}
\end{table}

\begin{table}[H]
\small
\centering
\extrarowheight=3pt
\begin{tabular}{c|c}
degree &$C^\flat$\\

\midrule[2pt]

$(-1,0,1)$
&$
q^2(q-q^{-1})^2 KF\otimes K^{-1}\otimes K^{-1}E
$
\\

\midrule[1pt]

$(0,-1,1)$

&$
q(q-q^{-1})(q^2-q^{-2}) K\otimes F\otimes K^{-1}E
-q(q-q^{-1})^2 \Lambda\otimes F\otimes K^{-1}E
$
\\

\midrule[1pt]

$(0,0,0)$

&$
\Lambda\otimes 1\otimes K^{-1}
+K\otimes 1\otimes \Lambda
-(q+q^{-1}) K\otimes 1\otimes K^{-1}
$
\\

\midrule[1pt]

$(1,-2,1)$
&
$
-q^3(q-q^{-1})^4E\otimes KF^2\otimes K^{-1}E
$
\\

\midrule[1pt]

$(1,-1,0)$

&$
q(q-q^{-1})(q^2-q^{-2}) E\otimes KF\otimes K^{-1}
-q(q-q^{-1})^2 E\otimes KF\otimes \Lambda
$
\\

\midrule[1pt]

$(1,0,-1)$

&$
(q-q^{-1})^2 E\otimes K \otimes F
$

\end{tabular}
\end{table}

\begin{table}[H]
\small
\centering
\extrarowheight=3.3pt
\begin{tabular}{c|c}
degree &$\alpha^\flat$\\

\midrule[2pt]

$(-1,0,1)$
&$
q^2(q-q^{-1})^2 KF\otimes 1\otimes \Lambda K^{-1} E
$\\

\midrule[1pt]

$(-1,1,0)$
&$
q^2(q-q^{-1})^2 KF\otimes K^{-1} E\otimes K^{-1}\Lambda
$
\\

\midrule[1pt]

$(0,-1,1)$
&$
q^2(q-q^{-1})^2 K\otimes KF\otimes \Lambda K^{-1} E
$
\\

\midrule[1pt]

\multirow{2}{*}{
$(0,0,0)$
}

&$
\Lambda\otimes K^{-1}\otimes K^{-1}\Lambda
+\Lambda\otimes \Lambda\otimes 1
+K\otimes K\otimes \Lambda^2
+K\otimes \Lambda\otimes K^{-1}\Lambda
$\\

&$
\;-\;(q+q^{-1})(
K\otimes K^{-1}\otimes K^{-1}\Lambda
+K\otimes K\otimes K^{-1}\Lambda)
$
\\

\midrule[1pt]

$(0,1,-1)$
&$
(q-q^{-1})^2 K\otimes E\otimes \Lambda F
$
\\

\midrule[1pt]

$(1,-1,0)$
&$
(q-q^{-1})^2 E\otimes F\otimes \Lambda K^{-1}
$\\

\midrule[1pt]

$(1,0,-1)$
&$
(q-q^{-1})^2 E\otimes 1\otimes \Lambda F
$
\end{tabular}
\end{table}

\begin{table}[H]
\small
\centering
\extrarowheight=3pt
\begin{tabular}{c|c}
degree &$\beta^\flat$\\

\midrule[2pt]

$(-1,0,1)$
&$
q^2(q-q^{-1})^2 \Lambda K F\otimes 1\otimes K^{-1} E
$
\\

\midrule[1pt]

$(-1,1,0)$
&$
q^2(q-q^{-1})^2 \Lambda K F\otimes K^{-1} E\otimes K^{-1}
$
\\

\midrule[1pt]

$(0,-1,1)$
&$
q^2(q-q^{-1})^2 \Lambda K\otimes KF\otimes K^{-1}E
$
\\

\midrule[1pt]

\multirow{2}{*}{
$(0,0,0)$
}

&$
\Lambda^2\otimes K^{-1}\otimes K^{-1}
+1\otimes \Lambda\otimes \Lambda
+K\Lambda\otimes K\otimes \Lambda
+K\Lambda\otimes \Lambda\otimes K^{-1}
$\\

&$
\;-\;(q+q^{-1})(
K\Lambda\otimes K^{-1}\otimes K^{-1}
+K\Lambda\otimes K\otimes K^{-1}
)
$
\\

\midrule[1pt]

$(0,1,-1)$
&$
(q-q^{-1})^2 \Lambda K \otimes E\otimes F
$
\\

\midrule[1pt]

$(1,-1,0)$
&$
(q-q^{-1})^2 \Lambda E\otimes F\otimes K^{-1}
$\\

\midrule[1pt]

$(1,0,-1)$
&$
(q-q^{-1})^2 \Lambda E\otimes 1\otimes F
$
\end{tabular}
\end{table}

\begin{table}[H]
\small
\centering
\extrarowheight=3pt
\begin{tabular}{c|c}
degree &$\gamma^\flat$\\

\midrule[2pt]

$(-1,0,1)$
&$
q^2(q-q^{-1})^2 K F\otimes \Lambda\otimes K^{-1} E
$\\

\midrule[1pt]

$(-1,1,0)$
&$
q^2(q-q^{-1})^2 KF\otimes \Lambda K^{-1} E\otimes K^{-1}
$
\\

\midrule[1pt]

$(0,-1,1)$
&$
q^2(q-q^{-1})^2 K\otimes \Lambda K F\otimes K^{-1} E
$
\\

\midrule[1pt]

\multirow{2}{*}{
$(0,0,0)$
}

&$
\Lambda\otimes K^{-1}\Lambda \otimes K^{-1}
+\Lambda\otimes 1\otimes \Lambda
+K\otimes K \Lambda\otimes \Lambda
+K\otimes \Lambda^2\otimes K^{-1}
$\\

&$
\;-\;(q+q^{-1})(
K\otimes K^{-1}\Lambda\otimes K^{-1}
+K\otimes K\Lambda\otimes K^{-1}
)
$
\\

\midrule[1pt]

$(0,1,-1)$
&$
(q-q^{-1})^2 K\otimes \Lambda E\otimes F
$
\\

\midrule[1pt]

$(1,-1,0)$
&$
(q-q^{-1})^2 E\otimes \Lambda F\otimes K^{-1}
$\\

\midrule[1pt]

$(1,0,-1)$
&$
(q-q^{-1})^2 E\otimes \Lambda\otimes F
$

\end{tabular}
\end{table}
\end{lem}
\begin{proof}
The tables for $A^\flat$, $B^\flat$, $\alpha^\flat$, $\beta^\flat$, $\gamma^\flat$ are immediate from Lemma \ref{lem:DeltaL}. To obtain the table for $C^\flat$, one may apply Lemma \ref{lem:EF} to express each homogeneous component of $(\Delta(\Lambda)\otimes 1)\cdot (1\otimes \Delta(\Lambda))$ and $(1\otimes \Delta(\Lambda))\cdot (\Delta(\Lambda)\otimes 1)$ as a linear combination of the corresponding basis given in Lemma \ref{lem:basisfoldU}.
\end{proof}

We are now in the position to prove Theorem \ref{thm:embedd}.

\medskip

\noindent{\it Proof of Theorem \ref{thm:embedd}.}
By Lemma \ref{lem:commu} with $n=1$ the element $\Delta_2(\Lambda)$ commutes with $1\otimes \Delta(\Lambda)$ and $\Delta(\Lambda)\otimes 1$.  Since $\Lambda$ is central in $\U$ each of $1\otimes 1\otimes \Lambda$, $\Lambda\otimes 1\otimes 1$, $1\otimes \Lambda\otimes 1$ commutes with all elements of $\U\otimes \U\otimes \U$. Concluding from the above comments,  each of $\alpha^\flat$, $\beta^\flat$, $\gamma^\flat$ commutes with $A^\flat$, $B^\flat$, $C^\flat$. To see the existence of $\flat$ it remains to verify that
\begin{eqnarray}
\frac{\alpha^\flat}{q+q^{-1}}
&=&
A^\flat+
\frac{qB^\flat C^\flat-q^{-1}C^\flat B^\flat}{q^2-q^{-2}},
\label{e:v1}
\\
\frac{\beta^\flat}{q+q^{-1}}
&=&
B^\flat+
\frac{qC^\flat A^\flat-q^{-1}A^\flat C^\flat}{q^2-q^{-2}},
\label{e:v2}\\
\frac{\gamma^\flat}{q+q^{-1}}
&=&
C^\flat+
\frac{qA^\flat B^\flat-q^{-1}B^\flat A^\flat}{q^2-q^{-2}}.
\label{e:v3}
\end{eqnarray}
Equation (\ref{e:v3}) is immediate from the construction of $A^\flat$, $B^\flat$, $C^\flat$, $\gamma^\flat$. To verify (\ref{e:v1}) and (\ref{e:v2}), one may utilize Lemmas \ref{lem:EF} and \ref{lem:ABCflat} to express each homogeneous component of the right-hand sides of (\ref{e:v1}), (\ref{e:v2}) as a linear combination of the corresponding basis given in Lemma \ref{lem:basisfoldU}. Then it is able to check that the corresponding homogeneous components of both sides of (\ref{e:v1}), (\ref{e:v2}) are equal to each other. After the tedious verification  (\ref{e:v1}), (\ref{e:v2}) follow. This shows the existence of $\flat$. The homomorphism $\flat$ is unique since the $\F$-algebra $\triangle_q$ is generated by $A$, $B$, $C$.
\hfill $\square$

\medskip

Recall the formula (\ref{e:Casimirtriangle}) for the Casimir element $\Omega$ of $\triangle_q$. 

\begin{thm}\label{thm:Casimir}
The image of the Casimir element $\Omega$ of $\triangle_q$ under $\flat$ is equal to
\begin{gather}\label{e:Casimirflat}
(q+q^{-1})^2
-\Lambda^2\otimes 1\otimes 1
-1\otimes \Lambda^2\otimes 1
-1\otimes 1\otimes \Lambda^2
-(\Lambda\otimes \Lambda\otimes \Lambda)\cdot \Delta_2(\Lambda)
-\Delta_2(\Lambda)^2.
\end{gather}
\end{thm}
\begin{proof}
By (\ref{e:Casimirtriangle}) the image of $\Omega$ under $\flat$ is
\begin{gather}\label{e:Casimirflat2}
q A^\flat B^\flat C^\flat
+q^2 (A^\flat)^2
+q^{-2}(B^\flat)^2
+q^2(C^\flat)^2
-qA^\flat\alpha^\flat
-q^{-1}B^\flat\beta^\flat
-qC^\flat\gamma^\flat.
\end{gather}
Similar to the idea of the verification for (\ref{e:v1}) and (\ref{e:v2}), we express each homogeneous component of (\ref{e:Casimirflat}), (\ref{e:Casimirflat2}) as a linear combination of the corresponding basis given in Lemma \ref{lem:basisfoldU}. After then we can check that the corresponding homogeneous components of (\ref{e:Casimirflat}), (\ref{e:Casimirflat2}) are equal to each other. The theorem follows.
\end{proof}

We end this subsection with a corollary of Theorem \ref{thm:embedd}.

\begin{cor}\label{cor:commu}
The image of $\flat$ is contained in the centralizer of $\Delta_2(\U)$ in $\U^{\otimes 3}$.
\end{cor}
\begin{proof}
By Lemma \ref{lem:commu} with $n=1$ the elements $A^\flat=\Delta(\Lambda)\otimes 1$ and $B^\flat=1\otimes \Delta(\Lambda)$ are in the centralizer of $\Delta_2(\U)$ in $\U^{\otimes 3}$. Since $\Lambda$ is in the center of $\U$ the element $C^\flat$ is in the centralizer of $\Delta_2(\U)$ in $\U^{\otimes 3}$ as well. Since the $\F$-algebra $\triangle_q$ is generated by $A$, $B$, $C$ the corollary follows.
\end{proof}

\subsection{The injectivity of the homomorphism $\flat$}\label{s:inj}

To see the injectivity of $\flat$, we equip the ring $\Z^3$ with the lexicographical order. Let $u$ denote a nonzero element of $\U\otimes \U\otimes \U$. We call a nonzero homogeneous component of $u$ the {\it leading component} of $u$ if each homogeneous component of $u$ of higher degree is zero. The {\it degree} of $u$ is meant to be the degree of the leading component of $u$.

\begin{lem}\label{lem:leading}
For all $i,j,k,r,s,t\in \N$ the following {\rm (i)}, {\rm (ii)} hold:
\begin{enumerate}
\item The degree of $(A^\flat)^{i}
(B^\flat)^{j}
(C^\flat)^{k}
(\alpha^\flat)^{r}
(\beta^\flat)^{s}
(\gamma^\flat)^{t}$ 
is 
$$
(i+k+r+s+t,j-i,-j-k-r-s-t).
$$

\item The leading component of 
$(A^\flat)^{i}
(B^\flat)^{j}
(C^\flat)^{k}
(\alpha^\flat)^{r}
(\beta^\flat)^{s}
(\gamma^\flat)^{t}$ 
is 
\begin{gather*}
\phi(i,j,k,r,s,t)
\times 
\left\{
\begin{array}{ll}
E^{i+k+r+s+t}\otimes E^{j-i}\otimes F^{-j-k-r-s-t} 
\qquad 
&\hbox{if $i\leq j$},
\\
E^{i+k+r+s+t}\otimes F^{i-j}\otimes F^{-j-k-r-s-t} 
\qquad 
&\hbox{if $i> j$},
\end{array}
\right.
\end{gather*}
where $\phi(i,j,k,r,s,t)$ is equal to 
$q^{2(i-j)k}(q-q^{-1})^{2(\max\{i,j\}+k+r+s+t)}$ times 
\begin{align*}\label{e:leading}
\Lambda^s
\otimes
\Lambda^{t} K^{k}
\prod_{h=1}^{\min\{i,j\}}
(
\Lambda
-q^{1-2(h-i)}K
-q^{2(h-i)-1}K^{-1}
)
\otimes
\Lambda^r.
\end{align*}
\end{enumerate}
\end{lem}
\begin{proof}
Observe from Lemma \ref{lem:ABCflat} that the degrees of $A^\flat$, $B^\flat$, $C^\flat$, $\alpha^\flat$, $\beta^\flat$, $\gamma^\flat$ are
\begin{gather*}
(1,-1,0),
\qquad
(0,1,-1),
\qquad
(1,0,-1),
\\
(1,0,-1),
\qquad
(1,0,-1),
\qquad
(1,0,-1),
\end{gather*}
respectively. Statement (i) is immediate from the observation. To obtain (ii) one may apply Lemma \ref{lem:EF}(ii) and Lemma \ref{lem:ABCflat}.
\end{proof}

We are ready to prove the injectivity of $\flat$.

\begin{thm}\label{thm:injective}
The homomorphism $\flat$ is injective.
\end{thm}
\begin{proof}
Suppose on the contrary that there exists a nonzero element $I$ in the kernel of $\flat$.
For any $i,j,k,r,s,t\in \N$ let $c(i,j,k,r,s,t)$ denote the coefficient of $A^{i} B^{j} C^{k} \alpha^{r} \beta^{s} \gamma^{t}$ in $I$ with respect to the basis of $\triangle_q$ given in Lemma~\ref{lem:PBWaw}. Let $S$ denote the set consisting of  all $6$-tuples $(i,j,k,r,s,t)$ of nonnegative integers with $c(i,j,k,r,s,t)\not=0$. Thus
\begin{gather*}
\sum_{(i,j,k,r,s,t)\in S}
c(i,j,k,r,s,t)
A^{i} B^{j} C^{k} \alpha^{r} \beta^{s} \gamma^{t}=I.
\end{gather*}
Applying $\flat$ to either side of the equality, we obtain that
\begin{gather}\label{e:Isharp}
\sum_{(i,j,k,r,s,t)\in S}
c(i,j,k,r,s,t)
(A^\flat)^{i}
(B^\flat)^{j}
(C^\flat)^{k}
(\alpha^\flat)^{r}
(\beta^\flat)^{s}
(\gamma^\flat)^{t}
=0.
\end{gather}
Furthermore, for all $m,n,p\in \Z$ let $S(m,n,p)$ denote the set consisting of those $(i,j,k,r,s,t)\in S$ with
$$
(m,n,p)=(i+k+r+s+t,j-i,-j-k-r-s-t).
$$
Equation (\ref{e:Isharp}) can be written as
\begin{gather}\label{e:Isharp2}
\sum_{m,n,p\in \Z}
\sum_{(i,j,k,r,s,t)\in S(m,n,p)}
c(i,j,k,r,s,t)
(A^\flat)^{i}
(B^\flat)^{j}
(C^\flat)^{k}
(\alpha^\flat)^{r}
(\beta^\flat)^{s}
(\gamma^\flat)^{t}
=0.
\end{gather}
By Lemma \ref{lem:leading}(i) each summand in the inner summation of (\ref{e:Isharp2}) is of degree $(m,n,p)$.
Since $I\not=0$ the finite set $S$ is nonempty.
Hence we may define 
$$
(M,N,P)=\max\{(m,n,p)\in \Z^3\,|\, S(m,n,p)\not=\emptyset\}.
$$
Recall from Lemma \ref{lem:leading}(ii) the elements $\phi(i,j,k,r,s,t)\in U_0\otimes U_0\otimes U_0$ for all $i,j,k,r,s,t\in \N$.
By Lemma \ref{lem:leading}(ii) and the maximality of $(M,N,P)$, equation (\ref{e:Isharp2}) implies that
\begin{gather}\label{e:zero}
\sum_
{
\substack{
(i,j,k,r,s,t)
\in
S(M,N,P)}
}
c(i,j,k,r,s,t)
\cdot
\phi(i,j,k,r,s,t)
=0.
\end{gather}

Now, fix an element $(i,j,k,r,s,t)\in S(M,N,P)$ with minimum value at $k-\min\{i,j\}$. Observe that the term $(i,j,k,r,s,t)$ contributes to the coefficient of
\begin{gather*}
\Lambda^{s}
\otimes
\Lambda^{t}
K^{k-\min\{i,j\}}
\otimes
\Lambda^{r}
\end{gather*}
in the left-hand side of (\ref{e:zero}). Suppose that $(i',j',k',r',s',t')\in S(M,N,P)$ also contributes to the coefficient. Then
\begin{align*}
&i+k+r+s+t=i'+k'+r'+s'+t'=M,
\\
&j-i=j'-i'=N,
\\
&j+k+r+s+t=j'+k'+r'+s'+t'=-P,
\\
&r=r', \qquad
s=s'.
\end{align*}
By the choice of $(i,j,k,r,s,t)$ we have
$k-\min\{i,j\}=k'-\min\{i',j'\}$ and $t=t'$.
Solving these equations yields that $(i',j',k',r',s',t')=(i,j,k,r,s,t)$. Therefore the coefficient of $\Lambda^{s}
\otimes
\Lambda^{t}
K^{k-\min\{i,j\}}
\otimes
\Lambda^{r}$ in the left-hand side of (\ref{e:zero}) is equal to $c(i,j,k,r,s,t)$ times 
$$
(-1)^{\min\{i,j\}}
q^{2(i-j)k+\min\{i,j\}(\min\{i,j\}-2i)}
(q-q^{-1})^{2(\max\{i,j\}+k+r+s+t)}.
$$
Since the coefficient is nonzero, the equation (\ref{e:zero})
gives a nontrivial $\F[K^{\pm 1}]$-algebraic dependence relation for
$$
\Lambda\otimes 1\otimes 1,
\qquad
1\otimes \Lambda\otimes 1,
\qquad
1\otimes 1\otimes \Lambda.
$$
This leads to a contradiction to Lemma \ref{lem:basisfoldU}. The theorem follows.
\end{proof}

We end this subsection with a corollary of Theorem \ref{thm:injective}.

\begin{cor}\label{cor:Hopfauto}
Let $\phi$ denote an $\F$-Hopf algebra automorphism of $\U$. For any $\F$-algebra automorphism $\varphi$ of $\triangle_q$ the following {\rm (i)}, {\rm (ii)} are equivalent:
\begin{enumerate}
\item $\flat\circ\varphi=(\phi\otimes \phi\otimes \phi)\circ \flat$. 

\item $\varphi=1$. 
\end{enumerate}
\end{cor}
\begin{proof}
By \cite[\S3.1.2, Proposition 6]{Klimyk} there exists a nonzero scalar $\alpha\in \F$  satisfying 
\begin{gather*}
\phi(E)=\alpha E,
\qquad 
\phi(F)=\alpha^{-1} F,
\qquad 
\phi(K)=K.
\end{gather*}
By (\ref{e:Lambda}) the element $\Lambda$ is invariant under $\phi$. Hence $1\otimes 1\otimes \Lambda$, $\Lambda\otimes 1\otimes 1$, $1\otimes \Lambda\otimes 1$ are invariant under $\phi\otimes \phi\otimes \phi$. By (\ref{e:coalghom}) the elements $\Delta(\Lambda)\otimes 1$, $1\otimes \Delta(\Lambda)$, $\Delta_2(\Lambda)$ are invariant under $\phi\otimes\phi \otimes \phi$. Therefore the elements $A^\flat$, $B^\flat$, $C^\flat$ are invariant under $\phi\otimes \phi\otimes \phi$ by Theorem \ref{thm:embedd}. In other words, the homomorphism $(\phi\otimes \phi\otimes \phi)\circ \flat$ sends $A$, $B$, $C$ to $A^\flat$, $B^\flat$, $C^\flat$ respectively.

(ii) $\Rightarrow$ (i): By the above comment the two $\F$-algebra homomorphisms $\flat$ and $(\phi\otimes \phi\otimes \phi)\circ \flat$ agree on the generators $A$, $B$, $C$ of $\triangle_q$. Therefore $\varphi=1$ satisfies (i).

(i) $\Rightarrow$ (ii): By Theorem \ref{thm:injective} any automorphism $\varphi$ of $\triangle_q$ with (i) must sends $A$, $B$, $C$ to $A$, $B$, $C$ respectively. Therefore $\varphi=1$. 
\end{proof}

\subsection{Hopf $*$-algebras of $\U$ and algebra involutions of $\triangle_q$}\label{s:hopf&AW}


Throughout this subsection, we assume that the underlying field is the complex number field $\C$ and let $-:\C\to \C$ denote the complex conjugation.

Recall that an {\it involution} $*$ is a function with $*\circ *=1$. A vector space $V$ over $\C$ with an involution $*:V\to V$ is called a {\it $*$-vector space} if
\begin{gather*}
(\lambda u+\mu v)^*=\overline{\lambda} u^*+\overline{\mu} v^*
\end{gather*}
for all $u,v\in V$ and $\lambda,\mu\in \C$. An algebra $\A$ over $\C$ with an involution $*:\A \to \A$ is called a {\it $*$-algebra} if $\A$ is a $*$-vector space and
$$
(xy)^*=y^* x^*
$$
for all $x,y\in\A$. 
Here $*$ is called the {\it algebra involution} of a $*$-algebra $\A$. 
A coalgebra $\A$ over $\C$ with an involution $*$ is called a {\it $*$-coalgebra} if the coalgebra $\A$ is a $*$-vector space and the comultiplication $\Delta$ and the counit $\epsilon$ of $\A$ satisfy
\begin{gather}
(*\otimes *)\circ \Delta=\Delta\circ *,
\label{e:Delta*}
\\
\epsilon\circ *=-\circ\epsilon.
\notag
\end{gather}
A bialgebra over $\C$ with an involution $*$ is called a {\it $*$-bialgebra} if the bialgebra is a $*$-algebra and a $*$-coalgebra. A Hopf algebra over $\C$ with an involution $*$ is called a {\it Hopf $*$-algebra} if the Hopf algebra is a $*$-bialgebra. 
A Hopf $*$-algebra $\A$ is said to be {\it equivalent} to a Hopf $*'$-algebra $\A'$ if there exists a Hopf algebra isomorphism $\phi:\A\to \A'$ such that
\begin{gather*}
\phi\circ *=*'\circ \phi.
\end{gather*}
The equivalence relations on $*$-algebras, $*$-coalgebras and $*$-bialgebras are defined in similar ways.

Let $\R$ denote the real number field.
Recall from \cite[\S3.1.4]{Klimyk} that the Hopf algebra of $\U$ has the following Hopf $*$-algebra structures up to equivalence:

\begin{enumerate}
[topsep=0.5em,itemsep=0.3em,partopsep=0.5em]

\item[(R1)]  $E^*=KF$, $F^*=EK^{-1}$, $K^*=K$
for $q\in \R$.

\item[(R2)] $E^*=-KF$, $F^*=-EK^{-1}$, $K^*=K$
for $q\in \R$.

\item[(R3)] $E^*=E$, $F^*=F$, $K^*=K$ for $|q|=1$.

\item[(R4)] $E^*=\sqrt{-1} KF$, $F^*=\sqrt{-1} EK^{-1}$, $K^*=K$ for
$q\in \sqrt{-1}\R$.

\item[(R5)] $E^*=-\sqrt{-1} KF$, $F^*=-\sqrt{-1} EK^{-1}$, $K^*=K$ for
$q\in \sqrt{-1}\R$.

\end{enumerate}
\setlist[enumerate,1]{leftmargin=2em}
The cases (R1)--(R3) correspond to the real forms $\mathfrak{su}_2$, $\mathfrak{su}_{1,1}$, $\mathfrak{sl}_2(\R)$ of  $\mathfrak{sl}_2(\C)$ respectively. Hence these Hopf $*$-algebras of $\U$ are denoted by $U_q(\mathfrak{su}_2)$, $U_q(\mathfrak{su}_{1,1})$, $U_q(\mathfrak{sl}_2(\R))$ respectively. The cases (R4), (R5) have no classical counterparts.

\begin{lem}\label{lem:Lambda*}
\begin{enumerate}
\item $\Lambda^*=\Lambda$ for each algebra involution $*$ of $\U$ given in {\rm (R1)--(R3)}.

\item $\Lambda^*=-\Lambda$ for each algebra involution $*$ of $\U$ given in {\rm (R4)}, {\rm (R5)}.
\end{enumerate}
\end{lem}
\begin{proof}
To see (i), (ii) one may apply each algebra involution $*$ of $\U$ given in (R1)--(R5) to (\ref{e:Lambda}) and simplify the resulting equation by using the defining relations of $\U$.
\end{proof}

\begin{thm}\label{thm:dag}
For any Hopf $*$-algebra of $\U$
there exists a unique algebra involution $\dag:\triangle_q\to \triangle_q$ such that
\begin{gather}\label{e:dag}
\flat\circ \dag
=(*\otimes *\otimes *)\circ \flat.
\end{gather}
\end{thm}
\begin{proof}
By Theorem \ref{thm:injective}, if the algebra involution $\dag$ of $\triangle_q$ exists then it is unique. In what follows we prove the existence of $\dag$. 

Assume that $*$ is one of the involutions of $\U$ given in (R1)--(R3). 
By Lemma \ref{lem:Lambda*} the elements $1\otimes 1\otimes \Lambda$, $\Lambda\otimes 1\otimes 1$, $1\otimes \Lambda\otimes 1$ are invariant under $*\otimes *\otimes *$. By (\ref{e:Delta*}) the elements $\Delta(\Lambda)\otimes 1$, $1\otimes \Delta(\Lambda)$, $\Delta_2(\Lambda)$ are invariant under $*\otimes *\otimes *$. Thus, the elements $A^\flat$, $B^\flat$, $\gamma^\flat$ are invariant under $*\otimes *\otimes *$ by Theorem~\ref{thm:embedd}. On the other hand, by Lemma \ref{lem:ABc} there exists a unique algebra involution $\dag:\triangle_q\to \triangle_q$ such that
\begin{gather*}
A^\dag=A,
\qquad
B^\dag=B,
\qquad
\gamma^\dag=\gamma.
\end{gather*}
Combining the above comments the involution $\dag$ of $\triangle_q$ satisfies 
(\ref{e:dag}).

Assume that $*$ is one of the involutions of $\U$ given in (R4), (R5). 
By Lemma \ref{lem:Lambda*} the involution $*\otimes *\otimes *$ sends each of $1\otimes 1\otimes \Lambda$, $\Lambda\otimes 1\otimes 1$, $1\otimes \Lambda\otimes 1$ to its negative. By (\ref{e:Delta*}) the involution $*\otimes *\otimes *$ also sends each of $\Delta(\Lambda)\otimes 1$, $1\otimes \Delta(\Lambda)$, $\Delta_2(\Lambda)$ to its negative. Thus, by Theorem~\ref{thm:embedd} the images of $A^\flat$, $B^\flat$, $\gamma^\flat$ under $*\otimes *\otimes *$ are $-A^\flat$, $-B^\flat$, $\gamma^\flat$ respectively. On the other hand, 
by Lemma \ref{lem:ABc} there exists a unique algebra involution $\dag:\triangle_q\to \triangle_q$ with
\begin{gather*}
A^\dag=-A,
\qquad
B^\dag=-B,
\qquad
\gamma^\dag=\gamma.
\end{gather*}
Combining the above comments the involution $\dag$ of $\triangle_q$ satisfies 
(\ref{e:dag}). 

The cases (R1)--(R5) include all Hopf $*$-algebras of $\U$ up to equivalence. Therefore the above two cases imply the existence of $\dag$ by Corollary \ref{cor:Hopfauto}. The theorem follows.
\end{proof}

\section{The $\U\otimes \U\otimes \U$-modules as $\triangle_q$-modules}\label{s:RW}

By Theorem \ref{thm:embedd} each $\U\otimes \U\otimes \U$-module can be considered as a $\triangle_q$-module via pulling back the homomorphism $\flat$. The aim of this section is to study the $3$-fold tensor products of irreducible Verma $\U$-modules and of finite-dimensional $\U$-modules from the viewpoint of $\triangle_q$-modules.

The outline of this section is as follows.
In \S\ref{s:CGVerma} (resp. \S\ref{s:CG}) we display the coupled bases of $2$-fold tensor products of irreducible Verma $\U$-modules (resp. finite-dimensional irreducible $\U$-modules). In \S\ref{s:RW2} (resp. \S\ref{s:RW1}) we apply the results of \S\ref{s:AWmodule} and \S\ref{s:CGVerma} (resp. \S\ref{s:CG}) to find a decomposition formula for $3$-fold tensor products of irreducible Verma $\U$-modules (resp. finite-dimensional irreducible $\U$-modules) into the direct sums of finite-dimensional irreducible $\triangle_q$-modules. 

\subsection{Coupled bases of tensor products of irreducible Verma $\U$-modules}\label{s:CGVerma}

Recall from \S\ref{s:Umodule} the Verma $\U$-modules $M(\lambda)$ and the action of $K$, $F$, $E$ on $M(\lambda)$ with respect to the $\F$-basis $\{m_i^{(\lambda)}\}_{i\in \N}$ of $M(\lambda)$.

\begin{lem}\label{lem:W}
Assume that $q$ is not a root of unity. Let $\kappa$ and $\lambda$ denote nonzero scalars in $\F$. For each $h\in \N$ let
$$
N(h)
$$
denote the eigenspace of $\Delta(K)$ on $M(\kappa)\otimes M(\lambda)$ associated with eigenvalue
$\kappa\lambda  q^{-2h}$.
Then the following {\rm (i)}, {\rm (ii)} hold:
\begin{enumerate}
\item $M(\kappa)\otimes M(\lambda)=\displaystyle{\bigoplus_{h\in \N} N(h)}$.

\item For each $h\in \N$ the $\F$-vector space $N(h)$ is of dimension $h+1$.
\end{enumerate}
\end{lem}
\begin{proof}
Since $q$ is not a root of unity the scalars $\{\kappa\lambda q^{-2h}\}_{h\in \N}$ are mutually distinct. Therefore
\begin{gather*}
\sum_{h\in \N} N(h)
\end{gather*}
is a direct sum.
For each $h\in \N$ the $\F$-vector space $N(h)$ contains $m_i^{(\kappa)}\otimes m_j^{(\lambda)}$
for all $i,j\in \N$ with $i+j=h$. Combined with the fact that $
\{m_i^{(\kappa)}\otimes m_j^{(\lambda)}\}_{i,j\in \N}$
are an $\F$-basis of $M(\kappa)\otimes M(\lambda)$, the lemma follows.
\end{proof}


A decomposition rule for the $\U$-module $M(\kappa)\otimes M(\lambda)$ is given below.

\begin{prop}\label{prop:CGformulaVerma}
Assume that $q$ is not a root of unity. Let $\kappa$ and $\lambda$ denote nonzero scalars in $\F$ with
\begin{gather*}
\kappa,
\,
\lambda,
\,
\kappa\lambda
\not\in
\{\pm q^{n}
\,|\,
n\in \N\}.
\end{gather*} 
For all $h\in \N$ let 
\begin{eqnarray}
m^{(\kappa,\lambda;h)}_0
&=&
\sum_{i=0}^h
(-1)^i \kappa^i q^{i(1-i)}
\prod_{\ell=1}^{h-i}
\frac{[\kappa;\ell-h]}{[\lambda;i+\ell-h]}
m^{(\kappa)}_i\otimes m^{(\lambda)}_{h-i},
\label{e:m0}
\\
m^{(\kappa,\lambda;h)}_k
&=&\frac{1}{[k]} \Delta(F) m^{(\kappa,\lambda;h)}_{k-1}
\qquad
\hbox{for all $k\in \N^*$}.
\label{e:mk}
\end{eqnarray}
Then the following {\rm (i)}, {\rm (ii)} hold:
\begin{enumerate}

\item The vectors
\begin{gather*}\label{e:coupledVerma}
m^{(\kappa,\lambda;h)}_k
\qquad
\hbox{for all $h,k\in \N$}
\end{gather*}
are an $\F$-basis of $M(\kappa)\otimes M(\lambda)$.

\item There exists a unique $\U$-module isomorphism
\begin{eqnarray*}
M(\kappa)\otimes M(\lambda)
&\to &
\bigoplus_{h\in \N}
M(\kappa\lambda q^{-2h})
\end{eqnarray*}
that sends
$m^{(\kappa,\lambda;h)}_k$
to
$m^{(\kappa\lambda q^{-2h})}_k$ for all $h,k\in \N$.
\end{enumerate}
\end{prop}
\begin{proof}
Fix an $h\in \N$. Recall the $\F$-subspace $N(h)$ of $M(\kappa)\otimes M(\lambda)$ from Lemma \ref{lem:W}.
Observe that
\begin{gather}\label{e:basisWh}
m^{(\kappa,\lambda;h-i)}_{i}
\in N(h)
\qquad
(0\leq i\leq h).
\end{gather}
We show that the $(h+1)$ vectors (\ref{e:basisWh}) are linearly independent over $\F$. Suppose on the contrary that there exist scalars $\{c_i\}_{i=0}^h$ in $\F$, not all zero, such that
\begin{gather}\label{e:ind}
\sum_{i=0}^h
c_i
m^{(\kappa,\lambda;h-i)}_{i}
=0.
\end{gather}
Let $j=\max\{j\,|\,c_j\not=0\}$. Applying $\Delta(E)^j$ to either side of (\ref{e:ind}) it follows that
\begin{gather}\label{e:nonzero}
c_j
\prod_{i=1}^{j}
[\kappa \lambda q^{2(j-h)};i-j]
m^{(\kappa,\lambda;h-j)}_{0}=0.
\end{gather}
Since $\kappa\lambda\not=\pm q^n$ for all $n\in \N$ the coefficient in the left-hand side of (\ref{e:nonzero}) is nonzero. Therefore the left-hand side of (\ref{e:nonzero}) is a nonzero vector, a contradiction.
This shows that (\ref{e:basisWh}) are linearly independent over $\F$. By Lemma \ref{lem:W}(ii) the vectors (\ref{e:basisWh}) are an $\F$-basis of $N(h)$. Combined with Lemma \ref{lem:W}(i) the statement (i) follows.

Let $M$ denote the $\U$-submodule of $M(\kappa)\otimes M(\lambda)$ generated by
$m^{(\kappa,\lambda;h)}_0$. 
A direct calculation yields that
\begin{gather*}
Em^{(\kappa,\lambda;h)}_0=0,
\qquad
K m^{(\kappa,\lambda;h)}_0=\kappa\lambda q^{-2h} m^{(\kappa,\lambda;h)}_0.
\end{gather*}
By Proposition \ref{prop:UPUqsl2} there exists a unique $\U$-module homomorphism $\phi:M(\kappa\lambda q^{-2h})\to M$ that maps $m^{(\kappa\lambda q^{-2h})}_0$ to $m^{(\kappa,\lambda;h)}_0$. By the construction of $M$ the homomorphism $\phi$ is surjective. Since $\kappa\lambda\not=\pm q^n$ for all $n\in \N$ and by Lemma~\ref{lem:irrVerma}, the Verma $\U$-module $M(\kappa\lambda  q^{-2h})$ is irreducible. Since $M\not=0$ it follows that $\phi$ is injective. Therefore $\phi$ is a $\U$-module isomorphism. By (\ref{e:mk}) the vector $m^{(\kappa,\lambda;h)}_k$ ($k\in \N$) is the image of $m^{(\kappa\lambda q^{-2h})}_k$ under $\phi$. Combined with (i) the statement (ii) follows.
\end{proof}

The $\F$-basis of $M(\kappa)\otimes M(\lambda)$ given in Proposition \ref{prop:CGformulaVerma}(i) is called the {\it coupled $\F$-basis} of $M(\kappa)\otimes M(\lambda)$. On the other hand the $\U$-module $M(\kappa)\otimes M(\lambda)$ has the $\F$-basis
\begin{gather}\label{e:canon2foldVerma}
m_i^{(\kappa)}\otimes m_j^{(\lambda)}
\qquad
\hbox{for all $i,j\in \N$}.
\end{gather}
For all $h,i,j,k\in \N$ let
\begin{gather*}
\begin{Large}
\left[
\begin{smallmatrix}
\kappa \; &\lambda \; &\kappa\lambda q^{-2h}
\\
i &j &k
\end{smallmatrix}
\right]
\end{Large}
\end{gather*}
denote the coefficient of $m^{(\kappa)}_i
\otimes
m^{(\lambda)}_j$ in $m^{(\kappa,\lambda;h)}_k$ with respect to the $\F$-basis
(\ref{e:canon2foldVerma})
of $M(\kappa)\otimes M(\lambda)$.

\begin{prop}\label{prop:CGVerma}
Assume that $q$ is not a root of unity. Let $\kappa$ and $\lambda$ denote nonzero scalars in $\F$ with
$$
\kappa,\,
\lambda,\,
\kappa\lambda
\not\in
\{\pm q^{n}
\,|\,
n\in \N\}.
$$
For all $h,i,j,k\in \N$ the following {\rm (i)}, {\rm (ii)} hold:
\begin{enumerate}
\item If $h+k\not=i+j$ then
$
\begin{Large}
\left[
\begin{smallmatrix}
\kappa \; &\lambda \; &\kappa\lambda q^{-2h}
\\
i &j &k
\end{smallmatrix}
\right]
\end{Large}
=0$.

\item If $h+k=i+j$ then
$
\begin{Large}
\left[
\begin{smallmatrix}
\kappa \; &\lambda \; &\kappa\lambda q^{-2h}
\\
i &j &k
\end{smallmatrix}
\right]
\end{Large}
$
is equal to
$$
q^{i(h+j)}
\sum_{r=0}^h
(-1)^{r-h}
q^{(r-h)(2h+k-1)}
\kappa^{h-r}
\lambda^{h-r-i}
{i\brack h-r}
{j\brack r}
\prod_{s=1}^r
\frac{[\kappa;s-h]}{[\lambda;1-s]}.
$$
\end{enumerate}
\end{prop}
\begin{proof}
By (\ref{e:m0}), for all $h,i,j\in\N$ we have 
\begin{gather}\label{e:initialCG}
\begin{Large}
\left[
\begin{smallmatrix}
\kappa \; &\lambda \; &\kappa\lambda q^{-2h}
\\
i &j &0
\end{smallmatrix}
\right]
\end{Large}
=
\left\{
\begin{array}{ll}
0 \qquad &\hbox{if $h\not=i+j$},
\\
(-1)^i \kappa^i q^{i(1-i)}
\displaystyle{
\prod_{\ell=1}^{j}
\frac{[\kappa;\ell-h]}{[\lambda;\ell-j]}
}
\qquad
&\hbox{if $h=i+j$}.
\end{array}
\right.
\end{gather}
Comparing the coefficients of $m^{(\kappa)}_i\otimes m^{(\lambda)}_j$ in both sides of (\ref{e:mk}) yields that
\begin{gather}\label{e:recCG}
[k]
\begin{Large}
\left[
\begin{smallmatrix}
\kappa \; &\lambda \; &\kappa\lambda q^{-2h}
\\
i &j &k
\end{smallmatrix}
\right]
\end{Large}
=\lambda^{-1} q^{2j} [i]
\begin{Large}
\left[
\begin{smallmatrix}
\kappa \; &\lambda \; &\kappa\lambda q^{-2h}
\\
i-1 &j &k-1
\end{smallmatrix}
\right]
\end{Large}
+
[j]
\begin{Large}
\left[
\begin{smallmatrix}
\kappa \; &\lambda \; &\kappa\lambda q^{-2h}
\\
i &j-1 &k-1
\end{smallmatrix}
\right]
\end{Large}
\end{gather}
for all $h\in \N$ and $i,j,k\in \N^*$. The lemma follows by solving the recurrence relation (\ref{e:recCG}) with initial values (\ref{e:initialCG}).
\end{proof}

\subsection{Three-fold tensor products of irreducible Verma $\U$-modules}\label{s:RW2}

In this subsection, we study the $3$-fold tensor products of irreducible Verma $\U$-modules from the viewpoint of $\triangle_q$-modules.

\begin{thm}\label{thm:RW2}
Assume that $q$ is not a root of unity.
Let $\kappa$, $\lambda$, $\mu$ denote nonzero scalars in $\F$ with
\begin{gather*}
\kappa,\,
\lambda,\,
\mu, \,
\kappa\lambda,\,
\lambda\mu,\,
\kappa\lambda\mu
\not\in
\{\pm q^n\,|\,n\in \N\}.
\end{gather*}
For all $d,k\in \N$ let
$$
M_k(d)
$$
denote the $\F$-subspace of $M(\kappa)\otimes M(\lambda)\otimes M(\mu)$ spanned by the simultaneous eigenvectors of $\Delta_2(K)$ and $\Delta_2(\Lambda)$ associated with eigenvalues $
\kappa\lambda\mu q^{-2(d+k)}$ and $\kappa\lambda\mu q^{1-2d}
+\kappa^{-1}\lambda^{-1}\mu^{-1} q^{2d-1}
$,
respectively. Then the following {\rm (i)--(iv)} hold:
\begin{enumerate}
\item $M(\kappa)\otimes M(\lambda)\otimes M(\mu)
=
\displaystyle{
\bigoplus_{d,k\in \N}
M_k(d)}$.

\item For all $d,k\in \N$ the $\F$-vector space $M_k(d)$ is a $(d+1)$-dimensional irreducible $\triangle_q$-module isomorphic to $V_d(a,b,c)$ where 
\begin{eqnarray*}
a &=& \kappa^{-1}\lambda^{-1}q^{d-1},
\\
b &=& \lambda^{-1}\mu^{-1}q^{d-1},
\\
c &=& \mu^{-1}\kappa^{-1}q^{d-1}.
\end{eqnarray*}

\item For all $d,k\in \N$ the action of $A$, $B$ on $M_k(d)$ is as a Leonard pair.

\item For all $d,k\in \N$ the following are equivalent:
\begin{enumerate}
\item $A$, $B$, $C$ act on $M_k(d)$ as a Leonard triple.

\item $B$, $C$ act on $M_k(d)$ as a Leonard pair.

\item $C$, $A$ act on $M_k(d)$ as a Leonard pair.

\item $C$ is diagonalizable on $M_k(d)$. 

\item $\mu\kappa\not\in \{\pm q^n\,|\, 0\leq n\leq 2d-2\}$.
\end{enumerate}    
\end{enumerate}
\end{thm}
\begin{proof}
Since $q$ is not a root of unity and $\kappa$, $\lambda$, $\kappa\lambda$, $\kappa\lambda\mu \not=\pm q^n$ for all $n\in \N$ we may apply Proposition \ref{prop:CGformulaVerma}(i) to see that
\begin{gather*}
u^{(h,d)}_k
=
\sum_{i\in \N}
\sum_{j\in \N}
\begin{Large}
\left[
\begin{smallmatrix}
\kappa\lambda q^{-2h} \; &\mu \; &\kappa\lambda\mu q^{-2d}
\\
i &j &k
\end{smallmatrix}
\right]
\end{Large}
m^{(\kappa,\lambda;h)}_i
\otimes
m^{(\mu)}_j
\qquad
\hbox{for all $d,h,k\in \N$ with $0\leq h\leq d$}
\end{gather*}
are an $\F$-basis of $M(\kappa)\otimes M(\lambda)\otimes M(\mu)$.
Similarly the vectors
\begin{gather*}
w^{(h,d)}_k
=
\sum_{i\in \N}
\sum_{j\in \N}
\begin{Large}
\left[
\begin{smallmatrix}
\kappa \; &\lambda\mu  q^{-2h}\; &\kappa\lambda\mu q^{-2d}
\\
i &j &k
\end{smallmatrix}
\right]
\end{Large}
m^{(\kappa)}_i
\otimes
m^{(\lambda,\mu;h)}_j
\qquad
\hbox{for all $d,h,k\in \N$ with $0\leq h\leq d$}
\end{gather*}
are an $\F$-basis of $M(\kappa)\otimes M(\lambda)\otimes M(\mu)$.
Observe that $\{u^{(h,d)}_k\}_{h=0}^d$ and  $\{w^{(h,d)}_k\}_{h=0}^d$ are contained in $M_k(d)$
for all $d,k\in \N$. Hence
$$
M(\kappa)
\otimes
M(\lambda)
\otimes
M(\mu)
=\sum_{d,k\in \N}
M_k(d).
$$
Since $q$ is not a root of unity and $\kappa\lambda\mu\not=\pm q^n$ for all $n\in \N$ the sum in the right-hand side is a direct sum. Therefore (i) follows. Moreover we have

\begin{lem}\label{lem:basisNkd}
For all $d,k\in \N$ the following {\rm (i)}, {\rm (ii)} hold:
\begin{enumerate}
\item $\{u^{(h,d)}_k\}_{h=0}^d$ are an $\F$-basis of $M_k(d)$.

\item $\{w^{(h,d)}_k\}_{h=0}^d$ are an $\F$-basis of $M_k(d)$.
\end{enumerate}
\end{lem}

Now fix $d,k\in \N$. Recall from Theorem \ref{thm:embedd} the images of $A$, $B$, $\alpha$, $\beta$, $\gamma$ under $\flat$. The $\F$-vector space $M_k(d)$ is invariant under $\alpha$, $\beta$, $\gamma$ by Lemma \ref{lem:CasVerma} and invariant under $A$, $B$ by Lemma \ref{lem:basisNkd}. Therefore $M_k(d)$ is a $\triangle_q$-module by Lemma \ref{lem:ABc}. More precisely, we have the following results:

\begin{lem}\label{lem:ABabcVerma}
For all $d,k\in \N$ the following {\rm (i)--(iii)} hold:
\begin{enumerate}
\item
For $0\leq h\leq d$ the vector $u^{(h,d)}_k$ is an eigenvector of $A$ associated with eigenvalue
\begin{gather}\label{e:thetaV}
\kappa\lambda q^{1-2h}
+
\kappa^{-1}\lambda^{-1} q^{2h-1}.
\end{gather}

\item
For $0\leq h\leq d$ the vector  $w^{(h,d)}_k$ is an eigenvector of $B$ associated with eigenvalue
\begin{gather}\label{e:theta*V}
\lambda\mu q^{1-2h}
+
\lambda^{-1}\mu^{-1} q^{2h-1}.
\end{gather}

\item The elements $\alpha$, $\beta$, $\gamma$ act on $M_k(d)$ as scalar multiplications by
\begin{gather}
(\kappa q+\kappa^{-1} q^{-1})
(\lambda q+\lambda^{-1} q^{-1})
+
(\mu q+\mu^{-1} q^{-1})
(\kappa\lambda\mu q^{1-2d}
+\kappa^{-1}\lambda^{-1}\mu^{-1} q^{2d-1}),
\label{e:wV}
\\
(\lambda q+\lambda^{-1} q^{-1})
(\mu q+\mu^{-1} q^{-1})
+
(\kappa q+\kappa^{-1} q^{-1})
(\kappa\lambda\mu q^{1-2d}
+\kappa^{-1}\lambda^{-1}\mu^{-1} q^{2d-1}),
\label{e:w*V}
\\
(\mu q+\mu^{-1} q^{-1})
(\kappa q+\kappa^{-1} q^{-1})
+
(\lambda q+\lambda^{-1} q^{-1})
(\kappa\lambda\mu q^{1-2d}
+\kappa^{-1}\lambda^{-1}\mu^{-1} q^{2d-1}),
\label{e:weV}
\end{gather}
respectively.
\end{enumerate}
\end{lem}

By Proposition \ref{prop:CGformulaVerma}(ii) we have
\begin{gather*}
\Delta_2(F)^k u_0^{(h,d)}=
[1][2]\cdots [k]
u_k^{(h,d)}
\qquad
(0\leq h\leq d).
\end{gather*}
By Lemma \ref{lem:basisNkd}(i) and since $q$ is not a root of unity the above equation induces an $\F$-linear isomorphism from $M_0(d)$ onto $M_k(d)$. By Corollary \ref{cor:commu} each element of $\triangle_q^\flat$ commutes with $\Delta_2(F)$. Hence the induced $\F$-linear isomorphism $M_0(d)\to M_k(d)$ is a $\triangle_q$-module isomorphism. By the above comment, we may consider the $\triangle_q$-module
$$
V=M_0(d)
$$
instead of $M_k(d)$.
Let $a$, $b$, $c$ be as in the statement (ii). Let $\{\theta_i\}_{i\in \Z}$, $\{\theta_i^*\}_{i\in \Z}$, $\{\varphi_i\}_{i\in \Z}$ and $\omega$, $\omega^*$, $\omega^\e$ denote the corresponding scalars defined in \S\ref{s:AWmodule} with $\nu=q^d$. It is straightforward to check that $\theta_h$, $\theta_h^*$, $\omega$, $\omega^*$, $\omega^\e$ coincide with (\ref{e:thetaV})--(\ref{e:weV}) respectively. By Lemmas \ref{lem:basisNkd}(i) and \ref{lem:ABabcVerma}(i) the element $A$ is diagonalizable on $V$ with eigenvalues $\{\theta_i\}_{i=0}^d$. Hence the characteristic polynomial of $A$ on $V$ is $\prod\limits_{i=0}^d(X-\theta_i)$. 
By Lemma \ref{lem:ABabcVerma}(ii) we have
\begin{gather}\label{e:Bwtheta0}
Bw^{(0,d)}_0=\theta_0^*w^{(0,d)}_0.
\end{gather}
Also, it follows from Lemma \ref{lem:ABabcVerma}(iii) that
\begin{gather*}
\alpha w^{(0,d)}_0=\omega w^{(0,d)}_0,
\qquad
\beta w^{(0,d)}_0=\omega^* w^{(0,d)}_0,
\qquad
\gamma w^{(0,d)}_0=\omega^\e w^{(0,d)}_0.
\end{gather*}
To apply Proposition \ref{prop:hom} it remains to verify that
\begin{gather}\label{e:BAphi1}
(B-\theta_1^*)(A-\theta_0) w^{(0,d)}_0=\varphi_1 w^{(0,d)}_0.
\end{gather}

To check (\ref{e:BAphi1}) we consider the $\F$-basis
\begin{gather}\label{e:canon3foldVerma}
m^{(\kappa)}_i\otimes m^{(\lambda)}_j\otimes m^{(\mu)}_k
\qquad
i,j,k\in \N
\end{gather}
of $M(\kappa)\otimes M(\lambda)\otimes M(\mu)$.
For all $0\leq h\leq d$ and $i,j,k\in \N$ let $c_h(i,j,k)$ denote the coefficient of $m^{(\kappa)}_i\otimes m^{(\lambda)}_j\otimes m^{(\mu)}_k$ in $w^{(h,d)}_0$ with respect to (\ref{e:canon3foldVerma}). By the construction of $w^{(h,d)}_0$ the coefficient 
\begin{gather*}
c_h(i,j,k)=
\sum_{\ell\in \N}
\begin{Large}
\left[
\begin{smallmatrix}
\kappa \; &\lambda\mu q^{-2h} \; &\kappa\lambda\mu q^{-2d}
\\
i &\ell &0
\end{smallmatrix}
\right]
\end{Large}
\begin{Large}
\left[
\begin{smallmatrix}
\lambda \; &\mu \; &\lambda\mu q^{-2h}
\\
j &k &\ell
\end{smallmatrix}
\right]
\end{Large}
\end{gather*}
for all $0\leq h\leq d$ and $i,j,k\in \N$.
Applying the expression of $A^\flat$ given in Lemma \ref{lem:ABCflat} the coefficient of $m^{(\kappa)}_d\otimes m^{(\lambda)}_0\otimes m^{(\mu)}_0$ in $Aw^{(0,d)}_0$ with respect to (\ref{e:canon3foldVerma}) is equal to 
$
a_0\cdot c_0(d,0,0)
+
a_1\cdot c_0(d-1,1,0)
$
where
\begin{align*}
a_0 &=\lambda^{-1}(\kappa q+\kappa^{-1} q^{-1})
+\kappa q^{1-2d}(\lambda-\lambda^{-1}),
\\
a_1 &=\kappa \lambda^{-1} q^{2(1-d)}
(\lambda-\lambda^{-1})
(q^d-q^{-d}).
\end{align*}
Here $c_0(d-1,1,0)$ is interpreted as an indeterminate if $d=0$. By Proposition \ref{prop:CGVerma} the scalars
\begin{align*}
c_0(d,0,0)&=(-1)^d \kappa^d q^{d(1-d)},
\qquad
c_h(d,0,0)=0
\qquad (1\leq h\leq d),
\\
c_0(d-1,1,0)&=(-1)^{d-1}
\kappa^{d-1}\mu^{-1}
q^{-(d-1)(d-2)}
\frac{\kappa q^{1-d}-\kappa^{-1} q^{d-1}}
{\lambda\mu-\lambda^{-1}\mu^{-1}}
\qquad
\hbox{if $d\geq 1$}.
\end{align*}
Comparing the coefficient of $m^{(\kappa)}_d\otimes m^{(\lambda)}_0\otimes m^{(\mu)}_0$ in $w^{(0,d)}_0$ and $(A-\theta_0)w^{(0,d)}_0$yields that the coefficient of $w^{(0,d)}_0$ in $(A-\theta_0)w^{(0,d)}_0$ with respect to the $\F$-basis $\{w^{(h,d)}_0\}_{h=0}^d$ of $V$ is equal to
\begin{gather}\label{e:coeffw0}
a_0-\theta_0
+
a_1
\frac{c_0(d-1,1,0)}{c_0(d,0,0)}.
\end{gather}
On the other hand, we apply $w^{(0,d)}_0$ to either side of the second relation shown in Lemma \ref{lem:ABc}. Simplifying the resulting equation by using (\ref{e:Bwtheta0}) we obtain that
\begin{gather}\label{e:Btrirelation}
(B-\theta_{-1}^*)
(B-\theta_0^*)
(B-\theta_1^*)
Aw^{(0,d)}_0=0.
\end{gather}
Since $q$ is not a root of unity and $\lambda\mu\not=\pm q^n$ for all $n\in \N$, the scalars $\{\theta_i^*\}_{i=0}^d$ are mutually distinct and $\theta_{-1}^*\not=\theta_i^*$ for all $2\leq i\leq d$. Combined with Lemmas \ref{lem:basisNkd}(ii) and \ref{lem:ABabcVerma}(ii) the element $B$ is diagonalizable on $V$ with simple eigenvalues $\{\theta_i^*\}_{i=0}^d$. Therefore $B-\theta_{-1}^*$ can be dropped from (\ref{e:Btrirelation}) even if $\theta_{-1}^*=\theta_0^*$ or $\theta_{-1}^*=\theta_1^*$. It follows from (\ref{e:Bwtheta0}) that 
$$
(B-\theta_1^*)Aw^{(0,d)}_0\in \F w^{(0,d)}_0.
$$
Hence the left-hand side of (\ref{e:BAphi1}) is a scalar multiple of $w^{(0,d)}_0$ by the scalar (\ref{e:coeffw0}) times $\theta_0^*-\theta_1^*$. To see (\ref{e:BAphi1}) it is now routine to check that $\varphi_1$ coincides with the aforementioned scalar.

Thanks to Proposition \ref{prop:hom} there exists a unique $\triangle_q$-module homomorphism $V_d(a,b,c)\to V$ that sends $v_0$ to $w^{(0,d)}_0$. By the assumptions on $\kappa$, $\lambda$, $\mu$, $q$ the scalars $a$, $b$, $c$, $q$ satisfy Lemma \ref{lem:irr}(i), (ii). Therefore the $\triangle_q$-module $V_d(a,b,c)$ is irreducible and the above $\triangle_q$-module homomorphism $V_d(a,b,c)\to V$ is injective.
Since $V$ and $V_d(a,b,c)$ have the same dimension $d+1$ over $\F$ the $\triangle_q$-module $V$ is isomorphic to $V_d(a,b,c)$.
Therefore (ii) holds.

We have seen that $A$ and $B$ are diagonalizable on $V$. 
Therefore (iii) is immediate from Lemma \ref{lem:lp}. Using Lemma \ref{lem:lt} it is routine to verify (iv). The theorem follows.
\end{proof}




\subsection{Coupled bases of tensor products of finite-dimensional irreducible $\U$-modules}\label{s:CG}

Recall from \S\ref{s:Umodule} the finite-dimensional $\U$-modules $V(n,\e)$ and the action of $K$, $F$, $E$  on $V(n,\e)$ with respect to the $\F$-basis $\{v_i^{(n,\e)}\}_{i=0}^n$ of $V(n,\e)$.

Observe that the $\U$-module $V(n,-1)$ is isomorphic to $V(0,-1)\otimes V(n,1)$ and $V(n,1)\otimes V(0,-1)$ for each $n\in \N$. Thus, it is enough to consider the $\U$-modules 
$$
V(n)=V(n,1)
\qquad 
\hbox{for all $n\in \N$}. 
$$
For notational convenience, we write
$v_i^{(n)}=v_i^{(n,1)}$
for all 
$0\leq i\leq n$. 
The proposition below states a similar result to Proposition \ref{prop:CGformulaVerma}. For a proof please refer to \cite[\S VII.7]{kassel}.

\begin{prop}\label{prop:CGformula}
Assume that $q$ is not a root of unity. 
Let $m,n\in \N$. For all integers $h$ with $0\leq h\leq \min\{m,n\}$ let
\begin{eqnarray*}
v^{(m,n;h)}_0
&=&
\sum_{i=0}^h
(-1)^{i}
q^{i(m-i+1)}
{m-i\brack h-i}
{n\brack h-i}^{-1}
 v_{i}^{(m)}\otimes v_{h-i}^{(n)}, \label{e:u0}
 \\
v^{(m,n;h)}_k
 &=&
\frac{1}{[k]} \Delta(F)v^{(m,n;h)}_{k-1}
 \qquad
 (1\leq k\leq m+n-2h). \label{e:uk}
\end{eqnarray*}
Then the following {\rm (i)}, {\rm (ii)} hold:
\begin{enumerate}
\item The vectors
\begin{gather*}\label{e:coupled}
v^{(m,n;h)}_k
\qquad
\hbox{for all
$0\leq h\leq \min\{m,n\}$
and
$0\leq k\leq m+n-2h$}
\end{gather*}
are an $\F$-basis of $V(m)\otimes V(n)$.

\item There exists a unique $\U$-module isomorphism
\begin{eqnarray*}
V(m)\otimes V(n)
&\to&
\bigoplus_{h=0}^{\min\{m,n\}}
V(m+n-2h)
\end{eqnarray*}
that sends $v^{(m,n;h)}_k$ to $v^{(m+n-2h)}_k$ for all
$0\leq h\leq \min\{m,n\}$ and
$0\leq k\leq m+n-2h$.
\end{enumerate}
\end{prop}

The $\F$-basis of $V(m)\otimes V(n)$ given in Proposition \ref{prop:CGformula}(i) is called the {\it coupled $\F$-basis} of $V(m)\otimes V(n)$. 
On the other hand, the $\U$-module $V(m)\otimes V(n)$ has the $\F$-basis
\begin{gather}\label{e:canon2fold}
v_i^{(m)}\otimes v_j^{(n)}
\qquad
\hbox{for all $0\leq i\leq m$ and
$0\leq j\leq n$}.
\end{gather}
For all $h,i,j,k\in \Z$ with $0\leq h\leq \min\{m,n\}$, $0\leq i\leq m$, $0\leq j\leq n$, $0\leq k\leq m+n-2h$ let
\begin{gather*}
\begin{Large}
\left[
\begin{smallmatrix}
m \; &n \; &m+n-2h
\\
i &j &k
\end{smallmatrix}
\right]
\end{Large}
\end{gather*}
denote the coefficient of
$v^{(m)}_i
\otimes
v^{(n)}_j$ in
$v^{(m,n;h)}_{k}$ with respect to the $\F$-basis (\ref{e:canon2fold})
of $V(m)\otimes V(n)$. By a similar argument to Proposition \ref{prop:CGformulaVerma} these coefficients can be expressed as below.

\begin{prop}\label{prop:CG}
Assume that $q$ is not a root of unity. Let $m,n\in \N$. For all $h,i,j,k\in \Z$ with $0\leq h\leq \min\{m,n\}$, $0\leq i\leq m$, $0\leq j\leq n$, $0\leq k\leq m+n-2h$ the following {\rm (i)}, {\rm (ii)} hold:
\begin{enumerate}
\item If $h+k\not=i+j$ then $
\left[
\begin{smallmatrix}
m \; &n \; &m+n-2h
\\
i &j &k
\end{smallmatrix}
\right]
=0$.

\item If $h+k=i+j$ then
$
\left[
\begin{smallmatrix}
m \; &n \; &m+n-2h
\\
i &j &k
\end{smallmatrix}
\right]
$
is equal to
\begin{gather*}
q^{i(h+j-n)}
\sum_{r=0}^{h}
(-1)^{h-r}
q^{(h-r)(m+n-2h-k+1)}
{m-h+r \brack r}
{i\brack h-r}
{j\brack r}
{n\brack r}^{-1}.
\end{gather*}
\end{enumerate}
\end{prop}

\subsection{Three-fold tensor products of finite-dimensional irreducible $\U$-modules}\label{s:RW1}

In the final subsection we show an analogue of Theorem \ref{thm:RW2} for the finite-dimensional $\U$-modules $V(n)$. Although the proof idea is similar to that of Theorem \ref{thm:RW2}, some places need nontrivial adjustments and hence a complete proof is included below.

\begin{lem}\label{lem:Tmnp}
Let $m$, $n$, $p$ denote any three real numbers. For any real numbers $h,\ell$ the following {\rm (i)--(iii)} are equivalent:
\begin{enumerate}
\item $0\leq h\leq \min\{m,n\}$ and
$0\leq \ell-h \leq \min\{m+n-2h,p\}$.

\item
$
\max\{0,\ell-p\}\leq h\leq \min\{m,n,\ell,m+n-\ell\}.
$

\item
$
\ell-\min\{p,\ell\}\leq h\leq \min\{m,\ell\}+\min\{n,\ell\}-\ell.
$
\end{enumerate}
\end{lem}
\begin{proof}
Rewriting the second inequality in (i) as  
$\ell-p\leq h\leq \min\{\ell,m+n-\ell\}$. 
Combining the inequality with the first inequality in (i) yields the equivalence of (i) and (ii). Since $\max\{0,\ell-p\}=\ell-\min\{p,\ell\}$ and $\min\{m,n,\ell,m+n-\ell\}=\min\{m,\ell\}+\min\{n,\ell\}-\ell$ the inequalities (ii) and (iii) are equivalent.
\end{proof}

\begin{lem}\label{lem:T}
Let $\ell$, $m$, $n$, $p$ denote any integers with $m,n,p\geq 0$. Then there exists an integer $h$ satisfying {\rm Lemma~\ref{lem:Tmnp}(i)--(iii)} if and only if one of the following {\rm (i)--(iii)} holds:
\begin{enumerate}

\item $2\ell\leq \min\{m,\ell\}+\min\{n,\ell\}+\min\{p,\ell\}$.

\item $\max\{m+n+p-\ell,m+\ell,n+\ell,p+\ell,2\ell\}\leq m+n+p$.

\item
$0\leq \ell\leq
\min\left\{m+n,n+p,p+m,
\frac{m+n+p}{2}
\right\}$.
\end{enumerate}
\end{lem}
\begin{proof}
Since $\ell$, $m$, $n$, $p$ are integers, there exists an integer $h$ satisfying Lemma~\ref{lem:Tmnp}(iii) if and only if (i) holds. Observe that $\min\{m,\ell\}+\min\{n,\ell\}+\min\{p,\ell\}-2\ell$ is equal to
$$
\min\{m,n,p,\ell,m+n-\ell,n+p-\ell,p+m-\ell,m+n+p-2\ell\}.
$$
Since $m$, $n$, $p$ are nonnegative, it follows that (i) holds if and only if
$$
\min\{\ell,m+n-\ell,n+p-\ell,p+m-\ell,m+n+p-2\ell\}\geq 0.
$$
The latter is equivalent to (ii). Therefore (i) and (ii) are equivalent.
Rewriting (ii) as an inequality about $\ell$ we obtain (iii). Therefore (ii) and (iii) are equivalent.
\end{proof}

After all this preparation let us prove the last result of this paper.

\begin{thm}\label{thm:RW}
Assume that $q$ is not a root of unity. Let $m$, $n$, $p$ denote any nonnegative integers. Let $\Sigma$ denote the set consisting of all pairs $(\ell,k)$ of integers satisfying
\begin{gather*}
0\leq
\ell
\leq
\min\left\{m+n,n+p,p+m,
\textstyle{
\frac{m+n+p}{2}
}
\right\},
\qquad
0\leq
k
\leq
m+n+p-2\ell.
\end{gather*}
For all $(\ell,k)\in \Sigma$ let
$$
V_k(\ell)
$$
denote the $\F$-vector subspace of $V(m)\otimes V(n)\otimes V(p)$ spanned by the simultaneous eigenvectors of $\Delta_2(K)$ and $\Delta_2(\Lambda)$ associated with eigenvalues $q^{m+n+p-2(k+\ell)}$ and $q^{m+n+p-2\ell+1}+q^{2\ell-m-n-p-1}$, respectively. 
Then the following {\rm (i)--(iii)} hold:
\begin{enumerate}
\item
$
V(m)\otimes V(n)\otimes V(p)
=
\displaystyle{
\bigoplus_{(\ell,k)\in \Sigma}
V_k(\ell)}$.

\item For all $(\ell,k)\in \Sigma$ the $\F$-vector space $V_k(\ell)$ is a finite-dimensional irreducible $\triangle_q$-module isomorphic to $V_d(a,b,c)$ where
\begin{eqnarray*}
a &=&
q^{\min\{m,\ell\}+\min\{n,\ell\}-\min\{p,\ell\}-m-n-1},
\\
b &=&
q^{\min\{n,\ell\}+\min\{p,\ell\}-\min\{m,\ell\}-n-p-1},
\\
c &=&
q^{\min\{p,\ell\}+\min\{m,\ell\}-\min\{n,\ell\}-p-m-1},
\\
d &=&
\min\{m,\ell\}+\min\{n,\ell\}+\min\{p,\ell\}-2\ell.
\end{eqnarray*}

\item For all $(\ell,k)\in \Sigma$ the action of $A$, $B$, $C$ on $V_k(\ell)$ is as a Leonard triple.
\end{enumerate}
\end{thm}
\begin{proof}
Since $q$ is not a root of unity we may apply Proposition \ref{prop:CGformula}(i) to see that 
\begin{gather*}
u_k^{(h,\ell)}
=
\sum_{i=0}^{m+n-2h}
\sum_{j=0}^p
\begin{Large}
\left[
\begin{smallmatrix}
m+n-2h \; &p \; &m+n+p-2\ell
\\
i &j &k
\end{smallmatrix}
\right]
\end{Large}
v^{(m,n;h)}_i
\otimes
v^{(p)}_j
\end{gather*}
for 
$0\leq h\leq \min\{m,n\}$, $0\leq \ell-h\leq \min\{m+n-2h,p\}$ and $0\leq k\leq m+n+p-2\ell$ form an $\F$-basis of $V(m)\otimes V(n)\otimes V(p)$.
Similarly the vectors
\begin{gather*}
w_k^{(h,\ell)}
=
\sum_{i=0}^{m+n-2h}
\sum_{j=0}^p
\begin{Large}
\left[
\begin{smallmatrix}
m \; &n+p-2h \; &m+n+p-2\ell
\\
i &j &k
\end{smallmatrix}
\right]
\end{Large}
v^{(m)}_i
\otimes
v^{(n,p;h)}_j
\end{gather*}
for 
$0\leq h\leq \min\{n,p\}$, $0\leq \ell-h\leq \min\{n+p-2h,m\}$
and $0\leq k\leq m+n+p-2\ell$ form an $\F$-basis of $V(m)\otimes V(n)\otimes V(p)$.
Rephrasing the ranges of $h$ and $\ell$ via Lemma~\ref{lem:Tmnp} it follows that 
\begin{align*}
u_k^{(h,\ell)}
\qquad
\begin{split}
\hbox{for all $\ell-\min\{p,\ell\}\leq h\leq
\min\{m,\ell\}+\min\{n,\ell\}-\ell$ and $0\leq k\leq m+n+p-2\ell$}
\end{split}
\end{align*}
are an $\F$-basis of $V(m)\otimes V(n)\otimes V(p)$. Similarly the vectors 
\begin{align*}
w_k^{(h,\ell)}
\qquad
\hbox{for all 
$\ell-\min\{m,\ell\}\leq h\leq
\min\{n,\ell\}+\min\{p,\ell\}-\ell$ 
and $0\leq k\leq m+n+p-2\ell$}
\end{align*}
are an $\F$-basis of $V(m)\otimes V(n)\otimes V(p)$. 
Observe that for all $(\ell,k)\in \Sigma$ the $\F$-vector space $V_k(\ell)$ contains 
\begin{alignat*}{2}
u_k^{(h,\ell)}
\qquad
&&\hbox{for all 
$\ell-\min\{p,\ell\}\leq h\leq
\min\{m,\ell\}+\min\{n,\ell\}-\ell$},
\\
w_k^{(h,\ell)}
\qquad
&&\hbox{for all 
$\ell-\min\{m,\ell\}\leq h\leq
\min\{n,\ell\}+\min\{p,\ell\}-\ell$}.
\end{alignat*}
Comparing the ranges of $h$ and $\ell$ via Lemma \ref{lem:T} yields that
\begin{gather*}
V(m)\otimes V(n)\otimes V(p)=\sum_{(\ell,k)\in \Sigma} V_k(\ell).
\end{gather*}
By the construction of $V_k(\ell)$ the sum in the right-hand side is a direct sum. Therefore (i) follows. Moreover we have

\begin{lem}\label{lem:M}
For all $(\ell,k)\in \Sigma$ the following {\rm (i)}, {\rm (ii)} hold:
\begin{enumerate}
\item
$
u_k^{(h,\ell)}$
for all 
$\ell-\min\{p,\ell\}\leq h\leq
\min\{m,\ell\}+\min\{n,\ell\}-\ell$
are an $\F$-basis of $V_k(\ell)$.

\item
$
w_k^{(h,\ell)}
$ for all 
$\ell-\min\{m,\ell\}\leq h\leq
\min\{n,\ell\}+\min\{p,\ell\}-\ell$
are an $\F$-basis of $V_k(\ell)$.
\end{enumerate}
\end{lem}

Now fix a pair $(\ell,k)\in \Sigma$. Recall from Theorem~\ref{thm:embedd} the images of $A$, $B$, $\alpha$, $\beta$, $\gamma$ under $\flat$.
The $\F$-vector space $V_k(\ell)$ is invariant under $\alpha$, $\beta$, $\gamma$ by Lemma \ref{lem:Casimir} and invariant under $A$, $B$ by Lemma \ref{lem:M}.  Therefore $V_k(\ell)$ is a $\triangle_q$-module by Lemma \ref{lem:ABc}. More precisely, we have the following results:

\begin{lem}\label{lem:eigen&abc}
For all $(\ell,k)\in \Sigma$ the following {\rm (i)--(iii)} hold:
\begin{enumerate}
\item For
$\ell-\min\{p,\ell\}\leq h\leq \min\{m,\ell\}+\min\{n,\ell\}-\ell$ the vector $u^{(h,\ell)}_k$ is an eigenvector of $A$ associated with eigenvalue
    \begin{gather}\label{e:theta}
    q^{m+n-2h+1}+q^{2h-m-n-1}.
    \end{gather}

\item For
$\ell-\min\{m,\ell\}\leq h\leq \min\{n,\ell\}+\min\{p,\ell\}-\ell$ the vector $w^{(h,\ell)}_k$ is an eigenvector of $B$ associated with eigenvalue
\begin{gather}\label{e:theta*}
q^{n+p-2h+1}+q^{2h-n-p-1}.
\end{gather}

\item The elements $\alpha$, $\beta$, $\gamma$ act on $V_k(\ell)$ as scalar multiplications by
\begin{gather}
(q^{m+1}+q^{-m-1})(q^{n+1}+q^{-n-1})
+(q^{p+1}+q^{-p-1})
(q^{m+n+p-2\ell+1}+q^{2\ell-m-n-p-1}),
\label{e:w}
\\
(q^{n+1}+q^{-n-1})(q^{p+1}+q^{-p-1})
+(q^{m+1}+q^{-m-1})
(q^{m+n+p-2\ell+1}+q^{2\ell-m-n-p-1}),
\label{e:w*}
\\
(q^{p+1}+q^{-p-1})(q^{m+1}+q^{-m-1})
+(q^{n+1}+q^{-n-1})
(q^{m+n+p-2\ell+1}+q^{2\ell-m-n-p-1}),
\label{e:we}
\end{gather}
respectively.
\end{enumerate}
\end{lem}

By Proposition \ref{prop:CGformula}(ii) we have
\begin{gather*}
\Delta_2(F)^k
u^{(h,\ell)}_0=
[1][2]\cdots[k]
u^{(h,\ell)}_k
\end{gather*}
for all $\ell-\min\{p,\ell\}\leq h\leq \min\{m,\ell\}+\min\{n,\ell\}-\ell$.
By Lemma \ref{lem:M}(i) and since $q$ is not a root of unity the above equation induces an $\F$-linear isomorphism from $V_0(\ell)$ onto $V_k(\ell)$.
By Corollary~\ref{cor:commu} each element of $\triangle_q^\flat$ commutes with $\Delta_2(F)$. Hence the induced $\F$-linear isomorphism $V_0(\ell)\to V_k(\ell)$ is a $\triangle_q$-module isomorphism. By the above comment we may consider the $\triangle_q$-module
$$
V=V_0(\ell)
$$
instead of $V_k(\ell)$.
Let $a$, $b$, $c$, $d$ be as in the statement (ii). 
Let $\{\theta_i\}_{i\in \Z}$, $\{\theta_i^*\}_{i\in \Z}$, $\{\varphi_i\}_{i\in \Z}$ and $\omega$, $\omega^*$, $\omega^\e$ be the corresponding scalars defined in \S\ref{s:AWmodule} with $\nu=q^d$. 
It is straightforward to check that $\theta_{h-\ell+\min\{p,\ell\}}$, $\theta_{h-\ell+\min\{m,\ell\}}^*$, $\omega$, $\omega^*$, $\omega^\e$ coincide with the scalars (\ref{e:theta})--(\ref{e:we}) respectively. By Lemma~\ref{lem:eigen&abc}(i) the element $A$ is diagonalizable on $V$ with eigenvalues $\{\theta_i\}_{i=0}^d$.
Hence the characteristic polynomial of $A$ on $V$ is equal to
$\prod\limits_{i=0}^d(X-\theta_i)$. 
For notational convenience we set
\begin{gather*}
s = \ell-\min\{m,\ell\},
\qquad 
w_h = w^{(h,\ell)}_0
\qquad
(s\leq h\leq s+d).
\end{gather*}
By Lemma~\ref{lem:eigen&abc}(ii) we have
\begin{gather}\label{e:Bwtheta0'}
Bw_s=\theta_0^*w_s.
\end{gather}
Also, it follows from Lemma~\ref{lem:eigen&abc}(iii) that
$$
\alpha w_s= \omega w_s,
\qquad
\beta w_s=\omega^* w_s,
\qquad
\gamma w_s=\omega^\e w_s.
$$
To apply Proposition \ref{prop:hom} it remains to verify that
\begin{gather}\label{e:con2}
(B-\theta_1^*)(A-\theta_0) w_{s}
=\varphi_1 w_{s}.
\end{gather}

To check (\ref{e:con2}) we consider the $\F$-basis
\begin{gather}\label{e:canbasis}
v_i^{(m)}\otimes
v_j^{(n)}\otimes
v_k^{(p)}
\qquad
(0\leq i\leq m,\, 0\leq j\leq n,\, 0\leq k\leq p)
\end{gather}
of $V(m)\otimes V(n)\otimes V(p)$. For all $s\leq h\leq s+d$ and $0\leq i\leq m$, $0\leq j\leq n$, $0\leq k\leq p$ let $c_h(i,j,k)$
denote the coefficient of
$v_i^{(m)}\otimes v_j^{(n)}\otimes v_k^{(p)}$ in $w_h$
with respect to (\ref{e:canbasis}).
By the construction of $w_h$ we have
\begin{gather*}
c_h(i,j,k)
=
\sum_{r=0}^{n+p-2h}
\begin{Large}
\left[
\begin{smallmatrix}
m \; &n+p-2h \; &m+n+p-2\ell
\\
i &r &0
\end{smallmatrix}
\right]
\end{Large}
\begin{Large}
\left[
\begin{smallmatrix}
n \; &p \; &n+p-2h
\\
j &k &r
\end{smallmatrix}
\right]
\end{Large}
\end{gather*}
for all $s\leq h\leq s+d$ and $0\leq i\leq m$, $0\leq j\leq n$, $0\leq k\leq p$. 
Applying the expression of $A^\flat$ given in Lemma~\ref{lem:ABCflat} the coefficient of $v_{\ell-s}^{(m)}\otimes v_0^{(n)}\otimes v_{s}^{(p)}$
in $Aw_{s}$ with respect to (\ref{e:canbasis}) is 
$
a_0
\cdot
c_{s}(\ell-s,0,s)
+
a_1
\cdot
c_{s}(\ell-s-1,1,s)$
where
\begin{eqnarray*}
a_0 &=&
q^{-n}(q^{m+1}+q^{-m-1})
+
q^{m-2\ell+2s+1}
(q^n-q^{-n}),
\\
a_1 &=&
q^{m-n-2\ell+2s+2}
(q^n-q^{-n})
(q^{\ell-s}-q^{s-\ell}).
\end{eqnarray*}
Here $c_{s}(\ell-s-1,1,s)$ is interpreted as an indeterminate if $\ell=s$.
By Proposition \ref{prop:CG} we have
\begin{align*}
c_{s}(\ell-s,0,s)
&=
(-1)^{(\ell-s)}
q^{(\ell-s)(m-\ell+s+1)}
{n \brack s}
{p \brack s}^{-1},
\\
c_{h}(\ell-s,0,s)
&=
0
\qquad
(s+1\leq h\leq s+d),
\end{align*}
and $c_{s}(\ell-s-1,1,s)$ is equal to
\begin{align*}
&(-1)^{(\ell-s)}
q^{(\ell-s-1)(m-\ell+s+2)}
\frac{[m-\ell+s+1]}{[n+p-2s]}
{n\brack s}
{p\brack s}^{-1}
\left(
q^n\frac{[s][p-s+1]}{[n]}
-
q^{2s-p}
\right)
\end{align*}
if $\ell>s$. Comparing the coefficient of $v_{\ell-s}^{(m)}\otimes v_0^{(n)}\otimes v_{s}^{(p)}$ in $w_s$ and $(A-\theta_0)w_s$ yields that the coefficient of $w_s$ in   $(A-\theta_0)w_{s}$ with respect to the $\F$-basis $\{w_h\}_{h=s}^{s+d}$ of $V$ is equal to
\begin{gather}\label{e:phi1}
a_0-\theta_0
+
a_1
\frac{c_{s}(\ell-s-1,1,s)}
{c_{s}(\ell-s,0,s)}.
\end{gather}
On the other hand, we apply $w_s$ to either side of the second relation in Lemma~\ref{lem:ABc}. Simplifying the resulting equation by using (\ref{e:Bwtheta0'}) it follows that
\begin{gather}\label{e:B3}
(B-\theta_{-1}^*)
(B-\theta_0^*)
(B-\theta_1^*)
Aw_{s}=0.
\end{gather}
Since $q$ is not a root of unity the scalars $\{\theta_i^*\}^d_{i=-1}$ are mutually distinct. Combined with Lemmas \ref{lem:M}(ii) and \ref{lem:eigen&abc}(ii) the element $B$ is diagonalizable on $V$ with simple eigenvalues $\{\theta_i^*\}^d_{i=0}$. 
Therefore $B-\theta_{-1}^*$ can be dropped from  (\ref{e:B3}). It follows from (\ref{e:Bwtheta0'}) that 
$$
(B-\theta_1^*)
Aw_{s}\in \F w_{s}.
$$
Hence the left-hand side of (\ref{e:con2}) is a scalar multiple of $w_{s}$ by the scalar (\ref{e:phi1}) times $\theta_0^*-\theta_1^*$. To see (\ref{e:con2}) it is now routine to check that $\varphi_1$ coincides with the aforementioned scalar.

Thanks to Proposition \ref{prop:hom} there exists a unique $\triangle_q$-module homomorphism $V_d(a,b,c)\to V$ that sends $v_0$ to $w_s$. Since $q$ is not a root of unity the scalars $a$, $b$, $c$, $q$ satisfy Lemma \ref{lem:irr}(i), (ii).
Therefore the $\triangle_q$-module $V_d(a,b,c)$ is irreducible and the above $\triangle_q$-module homomorphism $V_d(a,b,c)\to V$ is injective.
Since both of $V_d(a,b,c)$ and $V$ have the same dimension $d+1$ the $\triangle_q$-module $V$ is isomorphic to $V_d(a,b,c)$. Therefore (ii) holds. By the assumption that $q$ is not a root of unity again, the scalars $a$, $b$, $c$ fit  Lemma~\ref{lem:lt}(iv). Therefore (iii) holds. The theorem follows.
\end{proof}

\bibliographystyle{abbrv}
\bibliography{MP}

\end{document}